\documentclass[11pt]{amsart}

  \usepackage{setspace}
\makeatletter
\g@addto@macro\bfseries{\boldmath}
\makeatother

\usepackage{mathtools}

\usepackage{amsmath}
\usepackage{amssymb}
\usepackage[utf8]{inputenc}
\usepackage{amsthm}
\usepackage{hyperref}
\usepackage[margin=1in]{geometry}
\usepackage{enumitem}  
\usepackage{xcolor}

\newtheorem{theorem}{Theorem}[section]
\newtheorem{thm}{Theorem}

\newtheorem{lemma}[theorem]{Lemma}
\newtheorem{proposition}[theorem]{Proposition}

\theoremstyle{definition}

\newcommand{\supp}[1]{\mathrm{supp}({#1})} 
\renewcommand{\dim}{\mathrm{dim}} 
 
\newcommand{\Sym}[1]{\mathrm{Sym}({#1})}
\newcommand{\Sn}[1]{\mathrm{S}_{#1}}

\newcommand{\An}[1]{\mathrm{A}_{#1}}
\newcommand{\PSL}{\mathrm{PSL}} 
\newcommand{\Aut}{\mathrm{Aut}}

\newcommand{\diag}{\mathrm{diag}}

\newcommand{\PGamL}{\mathrm{P} \Gamma \mathrm{L}} 
\newcommand{\GamL}{\Gamma \mathrm{L}} 
\newcommand{\SigL}{\Sigma \mathrm{L}} 
\newcommand{\PSigL}{\mathrm{P} \Sigma \mathrm{L}}

\newcommand{\GL}{\mathrm{GL}} 
 
\newcommand{\PGL}{\mathrm{PGL}} 
\newcommand{\SL}{\mathrm{SL}}
\newcommand{\PG}{\mathcal{PG}} 
\newcommand{\F}{\mathbb{F}}  %
\newcommand{\Syl}[1]{\mathrm{Syl}} 
\newcommand{\id}{I} 
\newcommand{\I}{\mathrm{I}}
\newcommand{\RC}{\mathrm{RC}}
\newcommand{\Soc}{\mathrm{Soc}}
\renewcommand{\H}{\mathrm{H}}
\newcommand{\A}{A}%
\newcommand{\B}{B} %

\newcommand{\BB}{B} %

\renewcommand{\I}{\mathrm{I}}
\newcommand{\X}{X} %
\newcommand{\x}{x}
\newcommand{\Y}{Y}%
\newcommand{\y}{y}
\newcommand{\M}{\mathbb{M}}
\newcommand{\GGL}{G} %
\newcommand{\gp}{H} %
\newcommand{\ogp}{\overline{H}} %

\newcommand{\oD}{\overline{\Delta}} %

\renewcommand{\th}{\text{-th}}
\newcommand{\s}[2]{{\underset{#2,#1}{\sim}}}

\begin{document}
\title[The relational complexity of linear groups]{The relational complexity of linear groups\\acting on subspaces}
\author{Saul D. Freedman}
\address{\parbox{\linewidth}{
SAUL D. FREEDMAN, School of Mathematics and Statistics, University of St Andrews,\\
St Andrews, KY16 9SS, UK\\
\textit{Current address}: Centre for the Mathematics of Symmetry and Computation,\\The University of Western Australia, Crawley, WA 6009, Australia}\vspace{.2cm}}
\email{saul.freedman@uwa.edu.au}
\author{Veronica Kelsey}
\address{\parbox{\linewidth}{
VERONICA KELSEY, Department of Mathematics, University of Manchester,\\
Manchester, M13 9PL, UK}\vspace{.2cm}}
\email{veronica.kelsey@manchester.ac.uk}
\author{Colva M. Roney-Dougal}
\address{\parbox{\linewidth}{
COLVA M. RONEY-DOUGAL, School of Mathematics and Statistics, University of St Andrews,\\
St Andrews, KY16 9SS, UK}\vspace{.2cm}}
\email{colva.roney-dougal@st-andrews.ac.uk}


\maketitle

\begin{abstract} The relational complexity of a subgroup $G$ of
  $\Sym{\Omega}$ is a measure of the way in which the orbits of $G$
  on $\Omega^k$ for various $k$ determine the original action of $G$.
  Very few precise values of relational complexity are
  known. This paper determines the exact relational complexity of all
  groups lying between $\PSL_{n}(\F)$ and $\PGL_{n}(\F)$, for an
  arbitrary field $\F$,
  acting on the set of $1$-dimensional subspaces of $\F^n$. We also bound the relational complexity of all groups lying
  between $\PSL_{n}(q)$ and $\PGamL_{n}(q)$, and
 generalise these results to the action on
  $m$-spaces %
  for $m \ge 1$.
\end{abstract}

\section{Introduction}
\label{sec:intro}

The study of relational complexity began with work of Lachlan
in model theory as a way of studying \emph{homogeneous} relational
structures: those in which every isomorphism between induced
substructures extends to an automorphism of the whole structure. 
For the original definition see, for
example, \cite{Lach};
an equivalent definition in terms of permutation groups was given by
Cherlin \cite{CherlinGelfand}, and, apart from a slight
generalisation to group actions, is the one we now present.

Let $\Omega$ be an arbitrary set and let $\gp$ be a group acting on 
$\Omega$. Fix $k \in \mathbb{Z}$, and let $\X := (\x_1,\ldots,\x_k),\Y := (\y_1,\ldots,\y_k) \in
\Omega^k$.  For $r \le k$, we say that $\X$ and $\Y$ are
\emph{$r$-equivalent} under $\gp$, denoted $\X \s{r}{\gp} \Y$, if for
every $r$-subset of  indices $\{i_1, \ldots , i_r\} \subseteq \{1,\ldots,
k\}$, there exists an $h \in \gp$ such that $(\x_{i_1}^h,\ldots,
\x_{i_r}^h)=(\y_{i_1}, \ldots, \y_{i_r})$. If $\X \s{k}{\gp} \Y$,
i.e.~if $\Y \in \X^\gp$, then $\X$ and $\Y$ are \emph{equivalent}
under $\gp$. The \emph{relational complexity} of $\gp$, denoted
$\RC(\gp,\Omega)$, or $\RC(\gp)$ when $\Omega$ is clear, is the
smallest $r \ge 1$ such that $\X \s{r}{\gp} \Y$ implies $\Y \in \X^\gp$, for
all $\X,\Y \in \Omega^k$ and all $k \ge r$. Equivalently, $\RC(\gp)$
is the smallest $r$ such that $r$-equivalence of tuples implies
equivalence of tuples. Note that $\RC(\gp) \ge 2$ if $\gp \ne 1$ and $|\Omega| > 1$, as $\X$ or $\Y$ may contain repeated entries.

Calculating the precise relational complexity of a group is often very
difficult. A major obstacle  is that if $K < H \le \Sym{\Omega}$, 
then there is no uniform relationship between $\RC(K, \Omega)$ and $\RC(\gp, \Omega)$. For example, if $n \geq 4$, then the
relational complexities of the regular action of $C_n$ and natural actions of $\An{n}$ and $\Sn{n}$ are $2$, $n-1$ and $2$, respectively. 
In \cite{CherlinGelfand}, Cherlin gave three families of finite primitive
binary groups (groups with relational complexity two) and conjectured that this list was complete. In a dramatic recent breakthrough, this conjecture was proved by Gill, Liebeck and Spiga in \cite{binary}; this monograph also contains an extensive literature review. 

In \cite{CherlinGelfand, Cherlin16}, Cherlin determined the
  exact relational complexity of $\Sn{n}$ and $\An{n}$ in their
  actions on $k$-subsets of $\{1, \ldots, n\}$. The relational complexity of the remaining large-base primitive
groups is considered in \cite{CherlinMartinSaracino}.
Looking at finite primitive groups more generally, Gill,
Lodà and Spiga proved in \cite{GLS} that if $\gp \le \Sym{\Omega}$ is
primitive and not large-base, then $\RC(\gp, \Omega) < 9 \log |\Omega|
+ 1$ (our logarithms are to the base $2$). This bound was tightened by the second and third author in
\cite{KRD} to $5 \log |\Omega| +1$. Both \cite{GLS} and \cite{KRD}
bounded the relational complexity via base size, and the groups with
the largest upper bounds are classical groups acting %
on
subspaces of the natural module, and related product action
groups. This motivated us to obtain further information about the
relational complexity of these groups; this paper confirms that
  these bounds are tight, up to constants.

We now fix some notation for use throughout this paper. Let
$n$ be a positive integer, $\F$ be a (not necessarily finite) field, $V=\F^{n}$, and $\Omega_m = \PG_m(V)$ be the set of $m$-dimensional subspaces of $V$. We shall
study the relational complexity of the almost simple groups $\ogp$ with
${\PSL_n(\F) \trianglelefteq \ogp \leq \PGamL_n(\F)}$, acting on $\Omega_m$. We will generally work with the corresponding groups $H$ with ${\SL_n(\F) \trianglelefteq H \leq \GamL_n(\F)}$, as these naturally have the same relational complexity when acting on $\Omega_m$. 

Several of our results focus on %
the case $\F = \F_q$. %
 We begin with the following theorem of	 Cherlin.
\begin{theorem}[{\cite[Example
    3]{CherlinGelfand}}]\label{thm:CherlinNonzero} The relational
  complexity of $\GL_{n}(q)$ acting on the nonzero vectors of
  $\mathbb{F}_q^n$ is equal to $n$ when $q=2$, and $n+1$ when $q \geq
  3$. Hence also in the action on $1$-spaces we find that $\RC(\PGL_n(2), \Omega_1) = n$.
\end{theorem}

More generally, for $\PSL_n(q) \trianglelefteq
\ogp \le \PGL_n(q)$, Lodà \cite[Corollary 5.2.7]{Loda} shows that
$\RC(\ogp,\Omega_1) < 2 \log |\Omega_1| + 1$.  Other results imply
an alternative upper bound on $\RC(\ogp,\Omega_1)$. 
We first note that the \emph{height} of a permutation group $K$ on a set $\Omega$, denoted $\H(K)$ or $\H(K,\Omega)$, is the maximum size of a subset $\Delta$ of $\Omega$ with the property that $K_{(\Gamma)} < K_{(\Delta)}$ for each $\Gamma \subsetneq \Delta$. It is easy to show (see \cite[Lemma 2.1]{GLS}) that $\RC(K) \le \H(K) + 1$. By combining this with immediate generalisations of results of Hudson \cite[\S5.3-5.4]{Hudson} and Lod\`a \cite[Prop 5.2.1]{Loda}, we obtain the following (for $|\F| = 2$, see Theorem~\ref{thm:CherlinNonzero}; we also omit a few small exceptional cases for brevity). 

\begin{proposition}%
 \label{prop:lodabounds} Let $\PSL_n(\F) \trianglelefteq \ogp \leq \PGL_n(\F)$ and $|\F| \geq 3$.
  \begin{enumerate}[label=\rm{(\roman*)}]
    \item \label{proppart:scott} Suppose that $n=2$, with $|\F|
      = q \ge 7$ if $\ogp \ne \PGL_2(\F)$.
      If $|\F| \ge 4$, then
      $\H(\ogp,\Omega_1)=3$ and $\RC(\ogp,\Omega_1) = n+2 = 4$, whilst
      $\RC(\PGL_2(3), \Omega_1) = n = 2$.
  \item \label{proppart:bianca} If $n\geq 3$, then $\H(\ogp,\Omega_1)=2n-2$ and $\RC(\ogp,\Omega_1) \le 2n-1$.
  \end{enumerate}
\end{proposition}

Our first theorem gives the exact relational complexity of
groups between $\PSL_n(\F)$ and $\PGL_n(\F)$ for $n \ge 3$, acting naturally on 1-spaces. 

\begin{thm}\label{thm:A} Let $n \geq 3$, and let $\F$ be any field. Then the following hold.
\begin{enumerate}[label=\rm{(\roman*)}]
\item \label{A:PGL} $$\RC(\PGL_n(\F), \Omega_1)=\begin{cases} n & \text{ if }
    |\F| \le 3,\\ n+2 & \text{ if } |\F| \geq 4. \end{cases}$$
\item \label{A:notPGL} If  $\PSL_n(\F) \trianglelefteq \ogp < \PGL_n(\F)$, then 
$$\RC(\ogp, \Omega_1)=\begin{cases} 
2n-1 & \text{ if } n = 3,\\
2n-2 & \text{ if } n \geq 4. \end{cases}$$
\end{enumerate}
\end{thm}

For most groups, we see that the relational complexity is very close to the bound in Proposition~\ref{prop:lodabounds}\ref{proppart:bianca}. However,
the difference between the height and the relational complexity of $\PGL_n(\F)$ increases with $n$ when $|\F| \ge 3$. 
This addresses a recent question of Cherlin and Wiscons (see
\cite[p.~23]{binary}): there exists a family of finite primitive
groups that are not large-base, where the difference between height
and relational complexity can be arbitrarily large.  Theorem~\ref{thm:A} also provides infinitely many examples of almost simple groups $\ogp$ with $\RC(\Soc(\ogp)) >  \RC(\ogp)$. 

One way to
  interpret the gap between the relational complexity of
  $\PGL_n(\F)$ and its proper almost simple subgroups with socle $\PSL_n(\F)$
  is to observe that preserving linear dependence and indepedence is a
  comparatively ``local'' phenomenon, requiring information about the
 images of $n$-tuples of subspaces but not (very much) more, whereas
 restricting  determinants requires far richer information. This mimics the difference between the relational
  complexity of $\Sn{n}$ and $\An{n}$ in their natural actions, where
  requiring a map to be a permutation is ``local'', but requiring a
  permutation to be even is a ``global'' property.

We next bound the relational complexity of the remaining groups with socle $\PSL_n(q)$ that act on $\Omega_1$.  
For $k \in \mathbb{Z}_{> 0}$, the number of distinct prime divisors of $k$ is denoted by $\omega(k)$, with $\omega(1) = 0$.

\begin{thm}\label{thm:B}
Let $\ogp$ satisfy $\PSL_n(q) \leq \ogp \leq \PGamL_n(q)$, and let
$e:=|\ogp:\ogp \cap \PGL_n(q)|$. Suppose that $e > 1$, so that
$q \ge 4$ and $\ogp \not \le \PGL_n(q)$.
\begin{enumerate}[label=\rm{(\roman*)}]
\item \label{B:n=2} If $n=2$ and $q \geq 8$, then $4+\omega(e) \geq 
  \RC(\ogp, \Omega_1) \geq 4$, except that
    $\RC(\PSigL_2(9), \Omega_1) = 3$.
\item \label{B:ngeq3} If $n \geq 3$, then 
  $$2n-1+\omega(e) \geq \RC(\ogp, \Omega_1) \geq \begin{cases}
  n+2 & \text{ always}, \\
  n+3 & \text{ if } \PGL_n(q) < \ogp, \\ 
2n-2 & \text{ if } \ogp \le \PSigL_n(q) \ne \PGamL_n(q). \end{cases}$$
\end{enumerate}
\end{thm}

In fact, the lower bound of $2n-2$ holds for a larger family of groups; see
  Proposition~\ref{prop:PSLlower}. 

\begin{thm}\label{thm:Bm}
Let $\ogp$ satisfy $\PSL_n(q) \leq \ogp \leq \PGamL_n(q)$ and let
$e:=|\ogp:\ogp \cap \PGL_n(q)|$. 
 Fix $m \in \{2,\ldots, {\lfloor \frac{n}{2} \rfloor} \}$. Then
$$(m+1)n-2m+2+\omega(e) \geq \RC(\ogp,\Omega_m)\geq mn-m^2+1.$$
\end{thm}

GAP \cite{GAP} calculations using \cite{RComp} yield
$\RC(\PGamL_2(3^5),\Omega_1)=5 = 4 + \omega(5)$ and
$\RC(\PGamL_4(9),\Omega_1) = 8 = 7 + \omega(2)$, so the upper bounds of Theorem~\ref{thm:B}
cannot be improved in
general. On the other hand, $\RC(\PGamL_3(2^6),\Omega_1)$ achieves the lower bound of $6 =
3+3 < 7 = 5+\omega(6)$.  Additionally, $\RC(\PSL_4(2),\Omega_2)$ achieves the lower bound of $5$ from Theorem~\ref{thm:Bm}, while $\RC(\PSL_4(3),\Omega_2)=6$ and $\RC(\PGL_4(3),\Omega_2)= \RC(\PSL_4(4),\Omega_2) = \RC(\PGamL_4(4),\Omega_2) = 8$.

It is straightforward to use our results to bound the relational complexity in terms of the degree. For example, $\RC(\PGL_n(q), \Omega_1) <
\log(|\Omega_1|)+3$. Many of our arguments also apply to the case where $\F$ is an
arbitrary field; see Theorem~\ref{thm:lower_bound_n}, Lemmas~\ref{lem:PGamLn(q)}  and
\ref{lem:general_n+2}, and Propositions~\ref{prop:PSLlower} and \ref{prop:lowerm}.

This paper is structured as follows. In Section~\ref{sec:upper1}, we fix some more notation and prove some elementary lemmas, then prove upper bounds on the relational complexity of the relevant actions %
on $1$-spaces.  In Section~\ref{sec:lower1}, we shall prove
corresponding lower bounds, and then prove Theorems~\ref{thm:A} and
\ref{thm:B}. Finally, in Section~\ref{sec:mspace}, we
prove Theorem~\ref{thm:Bm}.%

\section{Action on 1-spaces: upper bounds}\label{sec:upper1}

In this section we present several preliminary lemmas, and then
determine upper bounds for the relational complexity of groups $\gp$,
with $\SL_n(\F) \trianglelefteq \gp \le \GL_n(\F)$,  acting on $\Omega_1$.

We begin with some notation that we will use throughout the remainder of the paper.
Let $\{e_1, \ldots, e_{n}\}$ be a basis for $V$. For a set $\Gamma$,
a tuple $\X = (\x_i)_{i = 1}^k \in \Gamma^k$ and a permutation $\sigma \in \Sn{k}$, we write $\X^\sigma$ to denote the $k$-tuple $(x_{1^{\sigma^{-1}}}, \ldots, x_{k^{\sigma^{-1}}})$. 
For a tuple $\X \in
\Omega_m^k$, we write $\langle \X \rangle$ to denote the subspace of $V$ spanned by all entries in $\X$. For  $i \in \{1, \ldots, k\}$, we shall write $(\X \setminus \x_i)$ to denote the subtuple of $\X$ obtained by deleting $\x_i$.

In the remainder of this section, let $\Omega:=\Omega_1 = \PG_1(V)$ and let $\gp$ be a group such that $\SL_n(\F) \trianglelefteq \gp \leq  \GL_n(\F)$.
Recall from Theorem \ref{thm:CherlinNonzero} that $\RC(\GL_n(\F),\Omega)
= n$ when $|\F| = 2$. Thus we shall assume throughout this section that $|\F| \ge 3$ and $n \ge 2$. 

We write $D$ to denote the subgroup of diagonal matrices of $\GL_n(\F)$ (with respect to the basis $\{e_1,\ldots, e_n\}$), and $\Delta := \big\{\langle e_i \rangle \mid i \in \{1,\ldots , n\} \big\}$. Observe that $D$ is nontrivial since $|\F| > 2$, and that $D
\cap \gp$ is the pointwise stabiliser $\gp_{(\Delta)}$.
For a vector $v = \sum_{i=1}^{n} \alpha_i e_i \in V$, the \emph{support} $\supp{v}$ of $v$ is the set $\{i \in \{1, \ldots, n\} \mid \alpha_i \ne 0\}$. Additionally, the \emph{support} $\supp{W}$ of a subset  $W$ of $V$ is the set $\bigcup_{w \in W} \supp{w}$, and similarly for tuples. 
In particular, $\Delta$ is the set of subspaces of $V$ with support of size $1$, and $\supp{W} = \supp{\langle W \rangle}$ for all subsets $W$ of $V$.

\subsection{Preliminaries}

We begin our study of the action of $\gp$ on $\Omega$ with a pair of lemmas that will enable us to consider only tuples of a very restricted form.

\begin{lemma}\label{lem:tuplefirstsubs-norm} Let $k \geq n$, and let $\X, \Y \in \Omega^k$ be such that $\X \s{n}{\gp} \Y$. Additionally, let $a:=\dim(\langle \X \rangle)$. Then there exist $\X'= (\x'_{1}, \ldots, \x'_{k}),\Y' = (\y'_{1}, \ldots, \y'_{k}) \in \Omega^k$ such that 
\begin{enumerate}[label=\rm{(\roman*)}]
\item $\x'_{i} = \y'_{i} = \langle e_i \rangle$ for $i \in \{1, \ldots, a\}$, and
\item $\X \s{r}{\gp} \Y$ if and only if $\X' \s{r}{\gp} \Y'$, for each $r \in \{1, \ldots, k\}$.
\end{enumerate}
\end{lemma}

\begin{proof}
  Observe that there exists $\sigma \in \Sn{k}$ such that
  $\langle \X^{\sigma}\rangle = \langle
  \x_{1^{\sigma^{-1}}}, \ldots, \x_{a^{\sigma^{-1}}} \rangle$. Since
  $\X \s{n}{\gp} \Y$ and $a \le n$, the definition of $a$-equivalence yields $\X^{\sigma} \s{a}{\gp} \Y^{\sigma}$. Hence there exists an $f \in \gp$ such that $x_{i^{\sigma^{-1}}}^f =  \y_{i^{\sigma^{-1}}}$ for all $i \in \{1, \ldots, a\}$, and so $\langle \Y^\sigma \rangle=\langle \y_{1^{\sigma^{-1}}}, \ldots, \y_{a^{\sigma^{-1}}} \rangle$.
Since $\SL_n(\F)$ is transitive on $n$-tuples of linearly independent $1$-spaces, there exists $h \in \SL_n(\F) \le H$ such that
$\x_{i^{\sigma^{-1}}}^{fh} = \y_{i^{\sigma^{-1}}}^h = \langle e_i \rangle$ for $i \in \{1, \ldots, a\}$. Define $\X', \Y' \in \Omega^k$ by $x_i'=x_{i^{\sigma^{-1}}}^{fh}$ and $y_i'=y_{i^{\sigma^{-1}}}^{h}$, so that $\X' = \X^{\sigma fh}$ and $\Y'  = \Y^{\sigma h}$. 
Then $\X'\s{r}{\gp} \Y'$ if and only if $\X^\sigma \s{r}{\gp} \Y^{\sigma}$, which holds if and only if $\X \s{r}{\gp} \Y$.	
\end{proof}

\begin{lemma}\label{lem:tuplefulldim} Let $k \ge r \geq n$, and let $\X, \Y \in \Omega^k$ be such that $\X\s{r}{\gp}\Y$. Additionally, let ${a :=\dim(\langle \X \rangle)}$ and assume that $a<n$. If $a = 1$, or if $\RC(\GL_{a}(\F),\PG_1(\F^a)) \le r$, then $\Y \in \X^\gp$.
\end{lemma}

\begin{proof}
  If $a = 1$, then all entries of $\X$ are equal, so since $r \ge n \ge 2$,
  we see that $\X \s{r}{\gp} \Y$  directly implies $\Y \in \X^\gp$. We
  will therefore suppose that $a \ge 2$ and
  $\RC(\GL_{a}(\F),\PG_1(\F^a))\le r$. By
  Lemma~\ref{lem:tuplefirstsubs-norm}, we may assume without loss of
  generality that $\langle \X \rangle = \langle \Y \rangle = \langle
  e_1 , \ldots, e_a \rangle$. As $\X \s{r}{\gp} \Y$ and
  $\RC(\mathrm{GL}_a(\F), \PG_1(\F^a))\le r$, there exists an element
  $g \in \mathrm{GL}_a(\F)$ mapping $\X$ to $\Y$, considered as tuples
  of subspaces of $\langle e_1 , \ldots, e_a \rangle$. We now let $h$
  be the diagonal matrix $\diag(\det({g}^{-1}), 1, \ldots 1) \in
  \GL_{n-a}(\F)$, and observe that $g \oplus h \in \SL_n(\F)$ maps
  $\X$ to $\Y$ and lies in $H$, since $\SL_n(\F)$ lies in $H$. Thus $\Y \in \X^\gp$.
\end{proof}

We now begin our study of some particularly nice $k$-tuples.

\begin{lemma}
\label{lem:eqsupps}
Let $k \geq n+1$, and let $\X, \Y \in \Omega^k$ be such that $\x_i= \y_i =
\langle e_i \rangle$ for $i \in \{1, \ldots, n\}$ and $\X \s{n+1}{\gp} \Y$. Then $\supp{\x_i} = \supp{\y_i}$ for all $i \in \{1,\ldots,k\}$.
\end{lemma}

\begin{proof}
It is clear that $\supp{\x_i}=\{i\}=\supp{\y_i}$ when $i \in \{1, \ldots, n\}$. Assume therefore that $i > n$. Since $\X \s{n+1}{\gp} \Y$, there exists $g \in \gp$ such that $(\langle e_1 \rangle , \ldots, \langle e_{n} \rangle, \x_i)^g=(\langle e_1 \rangle , \ldots, \langle e_{n} \rangle, \y_i)$. Observe that $g \in \gp_{(\Delta)} = D \cap \gp$, and so $\supp{\x_i} = \supp{\y_i}$.
\end{proof}

Our final introductory lemma collects several  elementary observations regarding %
  tuples of subspaces in $\oD := \Omega \setminus \Delta$, the set of $1$-dimensional subspaces of support size greater than $1$. 
For $r \ge 1$ and $\A, \B \in \oD^r$, we let $\M_{\A, \B}$ consist of all matrices in $\M_{n, n}(\F)$ that fix $\langle e_i \rangle$ for $1 \leq i \leq n$ and map $a_j$ into $b_j$ for $1 \leq j \leq r$. Notice that all matrices in $\M_{\A, \B}$ are diagonal, and  that if $g, h \in \M_{\A, \B}$, then $a_j^{g+h} = a_j^{g} + a_j^{h} \leq
  b_j$, so $\M_{\A, \B}$ is a subspace of $\M_{n, n}$.
For an $n \times n$ matrix $g = (g_{ij})$ and a subset $I$ of $\{1,\ldots,n\}$, we write $g|_I$ to denote the submatrix of $g$ consisting of the rows and columns with indices in $I$.

\begin{lemma}
\label{lem:submats}
Let $r \ge 1$, and let $\A =(a_1,\ldots, a_r), \B = (b_1, \ldots, b_r) \in \oD^{\, r}$. %
\begin{enumerate}[label=\rm{(\roman*)}]
\item \label{submats1} Let $a_i$ and $a_j$ be \textup{(}possibly equal\textup{)} elements of $\A$ such that $\supp{a_i} \cap \supp{a_j} \ne \varnothing$, and let $g \in \M_{\A, \A}$. Then $g|_{\supp{a_i,a_j}}$ is a scalar.
\item \label{dimofEr2} Suppose that $A \s{1}{D} \B$. Then for $1 \leq i \leq r$, the space
    $(\M_{(a_i), (b_i)})|_{\supp{a_i}} $ is one-dimensional, so the
    dimension of $\M_{(a_i), (b_i)}$ is equal to $n+1-|\supp{a_{i}}|$.
  \item \label{submats2} For a subtuple $\A'$ of $\A$, let $S := \{1, \ldots, n\} \setminus \supp{\A'}$. Then
    $\dim((\M_{\A', \A'})|_{S}) = |S|$.
\end{enumerate}
\end{lemma}

\begin{proof}
  Part~\ref{submats1} is clear.
   For Part~\ref{dimofEr2}, by assumption, there is an invertible
   diagonal matrix mapping $a_i$ to $b_i$, so $\supp{a_i} =
   \supp{b_i}$. Let $k$ and $\ell$ be distinct elements of $\supp{a_i}$,
   which exist as  $a_i \in \oD$. 
If $a_i^g \leq b_i$, then $e_{k}^g = \lambda e_{k}$ for some
$\lambda \in \F$, and the value of $\lambda$ completely determines the
image $\mu e_{\ell}$ of $e_{\ell}$ under $g$. The result
follows. Part~\ref{submats2} follows from the fact that
  for all $g \in \M_{A',A'}$, $s \in S$, $\lambda \in \F$ and $a \in
  A'$, the matrix obtained from $g$ by adding $\lambda$ to its $s$-th diagonal entry fixes $a$.
\end{proof}

\subsection{Upper bounds for $\SL_n(\F) \trianglelefteq \gp \le \GL_n(\F)$ on 1-spaces}

In this subsection, we will suppose that $n\geq 4$ and $|\F| \geq 3$, and let $\gp$ be any group such that ${\SL_n(\F) \trianglelefteq \gp \le \GL_n(\F)}$. Our main result is Theorem~\ref{thm:PSLupper}, which gives upper bounds on $\RC(\gp, \Omega)$. 

\begin{lemma}\label{lem:dim=n} If $\X \s{2n-2}{\gp} \Y$  implies that $\X \s{2n-1}{\gp} \Y$, for all  $\X, \Y \in \Omega^{2n-1}$ with $\x_i = \y_i = \langle e_i \rangle$ for $i \in \{1,\ldots,n\}$, then $\RC(H, \Omega) \le 2n-2$. 
\end{lemma}
\begin{proof} Let $k$ be at least $2n-1$, and let $\A, \B \in \Omega^{k}$ satisfy $\A \s{2n-2}{\gp} \B$. 
  Let $\A'$ be a subtuple of $\A$ of length $2n-1$, and $\B'$ the corresponding $(2n-1)$-subtuple of $\B$. We shall  show that $\B' \in {\A'}^{\gp}$ for all such $\A'$ and $\B'$, so that $\A \s{2n-1}{\gp}
  \B$.  It will then follow from Proposition~\ref{prop:lodabounds}\ref{proppart:bianca} that $\A \in \B^H$, as required. 

  Let $a:=\dim(\langle \A' \rangle)$, and suppose first that $a
  <n$. We observe from Proposition~\ref{prop:lodabounds}
    that if $a \ge 2$, then $\RC(\GL_a(\F), \PG_1(\F^a)) < 2n-2$. As $\A' \s{2n-2}{\gp} \B'$, Lemma~\ref{lem:tuplefulldim} yields $\B' \in {\A'}^{\gp}$.
 If instead $a=n$, then by Lemma~\ref{lem:tuplefirstsubs-norm} there exist  $\X$ and $\Y$ in $\Omega^{2n-1}$, such that $\x_i = \y_i = \langle e_i \rangle$ for each $i \in \{1,\ldots,n\}$, and such that for all $r \ge 1$, the relations  $\A' \s{r}{\gp} \B'$ and $\X \s{r}{\gp} \Y$ are equivalent. Now, $\A' \s{2n-2}{\gp} \B'$, so $\X \s{2n-2}{\gp} \Y$. If   $\X \s{2n-2}{\gp} \Y$  implies that $\X \s{2n-1}{\gp} \Y$, then $\B' \in {\A'}^{\gp}$, as required.
\end{proof}

We shall therefore let $\X$ and $\Y$ be elements of $\Omega^{2n-1}$
with $x_i = y_i = \langle e_i \rangle$ for $i \in \{1, \ldots, n\}$,
such that $\X \s{2n-2}{H} \Y$.
Additionally, for $i \in \{1, \ldots, 2n-1\}$ and $j \in \{1, \ldots, n\}$, define \begin{equation}\label{eq:alpha_{ij}}
  \alpha_{ij},\beta_{ij} \in \F \mbox{ so that } \x_i = \langle \sum_{j=1}^{n} \alpha_{ij}e_j \rangle \mbox{ and } \y_i = \langle \sum_{j=1}^{n} \beta_{ij}e_j \rangle. \end{equation}

\begin{lemma}\label{lem:combined}
With the notation above, if at least one of the following holds, then $\Y \in \X^\gp$.
\begin{enumerate}[label=\rm{(\roman*)}]
\item \label{part1:supportsubset} There exist $i,j \in \{n+1, \ldots, 2n-1\}$ with $i \neq j$ and $\supp{\x_i} \subseteq \supp{\x_j}$.
\item \label{part2:intersectionsize1} There exists a nonempty $R \subseteq \{n+1, \ldots, 2n-1\}$ with $|\bigcap_{i \in R} \supp{\x_i}|=1$.
\item \label{part4:supportsize4} There exists $i \in \{n+1, \ldots, 2n-1\}$ such that $\supp{\x_i} \geq 4$.
\end{enumerate}
\end{lemma}

\begin{proof} We begin by noting that Lemma~\ref{lem:eqsupps} yields $\supp{\y_i} = \supp{\x_i}$ for all $i \in \{1,\ldots,2n-1\}$.
\begin{enumerate}[label=\rm{(\roman*)}]
\item[\ref{part1:supportsubset}]
  Since $X \s{2n-2}{\gp} \Y$, there exists an $h \in H$ mapping $(\X \setminus x_i)$ to $(\Y \setminus y_i)$, and such an $h$ is necessarily diagonal,
  with fixed entries in $\supp{x_j}$ (up to scalar multiplication).
  
Now, let $\ell \in \{n+1,\ldots, 2n-1\} \setminus \{i,j\}$ (this is possible as $n \ge 4$). There exists an $h' \in H$ mapping $(\X \setminus x_\ell)$ to $(\Y \setminus y_\ell)$, and as before each such $h'$ is diagonal. Hence every matrix in $H \cap D$ mapping $x_j$ to $y_j$ maps $x_i$ to $y_i$, and in particular $x_i^h = y_i$ and so $\X^h = \Y$.
 
\item[\ref{part2:intersectionsize1}] 
  Let $\{\ell\}:=\bigcap_{i \in R} \supp{\x_i}$. Then $\alpha_{i\ell}
  \neq 0$ for all $i \in R$. Since $\X \s{2n-2}{\gp} \Y$, there exists
  $h \in\gp$ such that $(\X \setminus \x_\ell )^h =( \Y \setminus
  \y_\ell )$. It follows that for all $k \in \{1,\ldots, n\}
  \backslash \{\ell\}$, there exists $\gamma_k \in \F^*$ such that
  $e_k^h= \gamma_ke_k$. Thus for each $i \in R$,
\begin{equation*}\y_i = \x_i^h = \Big\langle \sum_{k \in \supp{\x_i}} \alpha_{ik} e_k^h \Big\rangle  = \Big\langle \alpha_{i\ell} e_\ell^h+\sum_{k \in \supp{\x_i} \setminus \{\ell\} } \alpha_{ik} \gamma_k e_k \Big\rangle.
\end{equation*}
Since $\alpha_{i\ell} \ne 0$, %
we deduce that $\supp{e_\ell^h} \subseteq \supp{\y_i} = \supp{\x_i}$. 
As this holds for all $i\in R$, we obtain $\supp{e_\ell^h} = \{\ell\}$.
Thus $\x_\ell^h = \langle e_\ell \rangle ^h = \langle e_\ell \rangle = \y_\ell$,  so $\X^h=\Y$.

\item[\ref{part4:supportsize4}] Permute the last $n-1$ coordinates of
  $\X$ and $\Y$ so that $\supp{\x_{n+1}} \geq 4$. By
  \ref{part2:intersectionsize1}, we may assume that $x_{i} \not \in \Delta$ for all $i \ge n+1$.
  For each $k \in
  {\{n+1, \ldots, 2n-1\}}$, we define $\X_{n+1}^k := (x_{n+1}, \ldots, x_k)$ and $\Y_{n+1}^k := (y_{n+1}, \ldots, y_k)$. 
  As $\supp{x_i} = \supp{y_i}$ for all $i$, we see that
    $X_{n+1}^k \s{1}{D} Y_{n+1}^k$, so
   $\X_{n+1}^k$ and $\Y_{n+1}^k$ satisfy the conditions of
   Lemma~\ref{lem:submats}\ref{dimofEr2}.

  Suppose first that there exists $j \in {\{n+2,\ldots, 2n-1\}}$ such that
  $\M_{\X_{n+1}^j, \Y_{n+1}^j} = \M_{\X_{n+1}^{j-1}, \Y_{n+1}^{j-1}}$. 
As $X \s{2n-2}{\gp} \Y$, there exists $h \in\gp \cap D$ such that $(\X
\setminus \x_j)^h=(\Y \setminus \y_j)$. Hence $h \in
\M_{\X_{n+1}^{j-1}, \Y_{n+1}^{j-1}}$, and so $h \in \M_{\X_{n+1}^{j}, \Y_{n+1}^{j}}$.
Therefore, $x_j^h = y_j$ and $\X^h = \Y$.

Hence we may assume instead that $\M_{\X_{n+1}^j, \Y_{n+1}^j} < \M_{\X_{n+1}^{j-1}, \Y_{n+1}^{j-1}}$ for all $j \in {\{n+2,\ldots, 2n-1\}}$. 
Then $\dim(\M_{\X_{n+1}^j, \Y_{n+1}^j}) \le \dim(\M_{\X_{n+1}^{j-1}, \Y_{n+1}^{j-1}})-1$. Lemma~\ref{lem:submats}\ref{dimofEr2} yields $\dim(\M_{\X_{n+1}^{n+1}, \Y_{n+1}^{n+1}}) \le {n-3}$, and hence $\M_{\X_{n+1}^{2n-2}, \Y_{n+1}^{2n-2}}  = \{0\} = \M_{\X_{n+1}^{2n-1}, \Y_{n+1}^{2n-2}}$, contradicting our assumption.
 \qedhere
\end{enumerate}
\end{proof}

We now prove the main result of this subsection.

\begin{theorem}\label{thm:PSLupper} Suppose that $n \geq 4$ and $|\F| \geq 3$, and let $\gp$ be any group with ${\SL_n(\F) \trianglelefteq \gp \le \GL_n(\F)}$.
  Then $\RC(\gp, \Omega) \le 2n-2$.
\end{theorem}

\begin{proof}
  Let $\X, \Y \in \Omega^{2n-1}$ be as defined before Lemma~\ref{lem:combined}. %
  By Lemma \ref{lem:dim=n} it suffices to show that $\Y \in \X^\gp$, so assume otherwise.
We may also assume that all subspaces in $\X$ are distinct, so that $|\supp{\x_i}| \in \{2, 3\}$ for each $i \in \{n+1, \ldots, 2n-1\}$ by  Lemma~\ref{lem:combined}\ref{part4:supportsize4}.

For $k \in \{2,3\}$, let $R_k$ be the set of all $i \in \{n+1,\ldots,2n-1\}$ such that $|\supp{\x_i}|=k$. Then
\begin{equation}\label{eqn:r2r3sum}
|R_2|+|R_3|=n-1.
\end{equation}
Observe from Lemma~\ref{lem:combined}\ref{part1:supportsubset}--\ref{part2:intersectionsize1} that if $i\in R_2$, then $\supp{\x_i} \cap \supp{\x_j} = \varnothing$ for each $j \in {\{n+1,\ldots,2n-1\} \setminus \{i\}}$. Hence $2|R_2| \le n$ and 
\begin{equation}\label{eqn:usize}
|U|:=\bigg|\bigcup_{j \in R_3}\supp{\x_j}\bigg| \le \bigg|\{1, \ldots,n\} \setminus \Big( \dot{\bigcup_{i\in R_2}} \supp{\x_i} \Big)\bigg| = n-2|R_2|.
\end{equation}

Observe next that $|R_3| \ge 1$, else $|R_2| = n-1$ by
\eqref{eqn:r2r3sum}, contradicting $2|R_2| \le n$. We shall determine
an expression for $|U|$ involving $|R_3|$, by considering the maximal
subsets $P$ of $R_3$ that correspond to pairwise overlapping
supports. To do so, define a relation $\sim$ on $R_3$ by $i \sim j$ if $\supp{\x_i} \cap \supp{\x_j} \ne \varnothing$, let $\mathcal{P}$ be the set of equivalence classes of the transitive closure of $\sim$, and let $P \in \mathcal{P}$. We claim that $|\bigcup_{c\in P} \supp{x_{c}}| = 2+|P|$.
By Lemma~\ref{lem:combined}\ref{part1:supportsubset}--\ref{part2:intersectionsize1}, $|\supp{\x_i} \cap \supp{\x_j}| \in \{0,2\}$ for all distinct $i,j \in R_3$. Thus our claim is clear if $|P| \in \{1,2\}$. 

If instead $|P| \geq 3$, then there exist distinct $c_1,c_2,c_3 \in P$ with $c_1 \sim c_2$ and $c_2 \sim c_3$. Let $I:=\bigcap_{i=1}^3 \supp{\x_{c_i}}$.
We observe that $|I|\neq 0$, and so Lemma~\ref{lem:combined}\ref{part2:intersectionsize1} shows that $I$ has size two and is equal to $\supp{x_{c_1}} \cap \supp{x_{c_3}}$. Hence $c_1 \sim c_3$  and $\bigcup_{i=1}^3 \supp{x_{c_i}}= I \dot{\cup} \big( \dot{\bigcup}_{i=1}^3 (\supp{x_{c_i}} \setminus I) \big)$.
If $|P|>3$, then there exists $c_4 \in P \setminus \{c_1,c_2,c_3\}$
such that, without loss of generality, $c_4 \sim c_1$. As $c_1 \sim c_j$ for each $j \in \{2,3\}$, the above argument shows that $\bigcap_{i \in \{1,j,4\}} \supp{\x_{c_i}} =I$ and $\bigcup_{i=1}^4 \supp{\x_{c_i}} = I \dot{\cup} \big( \dot{\bigcup}_{i=1}^4 (\supp{x_{c_i}} \setminus I) \big)$. Repeating this argument inductively on $|P|$ shows that ${\bigcup}_{c \in P} \supp{\x_{c}} = I  \dot{\cup} \big( \dot{\bigcup}_{c \in P} (\supp{x_{c}} \setminus I) \big)$, which has size $2+|P|$, as claimed.

Finally, let $r \ge 1$ be the number of parts of $\mathcal{P}$.
As $|R_3|=\sum_{P \in \mathcal{P}} |P|$,  it follows from our claim that $|U| = 2r + |R_3| \geq 2+ |R_3|$. Thus \eqref{eqn:usize} yields $2+|R_3| \leq n-2|R_2|$. Hence $2|R_2|+|R_3| \le n-2<n-1$, which is equal to $|R_2|+|R_3|$ by \eqref{eqn:r2r3sum}, a contradiction.
\end{proof}

\subsection{Upper bounds for $\GL_{n}(\F)$ on 1-spaces}
\label{subsec:pglupper}

In this subsection, we determine a much smaller upper bound on $\RC(\GL_{n}(\F),\Omega)$ via our main result, Theorem~\ref{thm:relupper}. We shall assume throughout that $n$ and $|\F|$ are at least $3$, and write $\GGL:=\GL_{n}(\F)$. 
Since $D$ is the pointwise stabiliser of $\Delta$ in $\GGL$, we will prove Theorem~\ref{thm:relupper} by combining Lemmas~\ref{lem:tuplefirstsubs-norm} and \ref{lem:tuplefulldim} with information about the action of $D$ on
$r$-tuples $\A$ and $\B$ of subspaces in $\oD = \Omega \setminus \Delta$. If these tuples are $(r-1)$-equivalent under $D$, then by acting on one with a suitable element of $\oD$ we may assume that their first $r-1$ entries are equal.
We shall denote the nonzero entries of elements $g$ of $D$ by just
$g_1, \ldots, g_n$ rather than $g_{11}, \ldots, g_{nn}$, since
$g$ is necessarily diagonal.

\begin{lemma}
\label{lem:diaginequiv}
Let $r \ge 3$, and let $\A,\B \in \oD^{\,r}$ be such that $(a_1,\ldots,a_{r-1}) = (b_1,\ldots,b_{r-1})$, $\A \s{r-1}{D} \B$, and $\B \notin \A^D$. Let $C = \{a_1,\ldots,a_{r-1}\}$ and assume also 
that $\supp{C} = \{1, \ldots, n\}$. Then \textup{(}after reordering the basis for $V$ and $(a_1, \ldots, a_{r-1})$  if necessary\textup{)} the following statements hold.
\begin{enumerate}[label=\rm{(\roman*)}]
\item \label{diaginequiv2} There exist integers $2 \le i_1 < i_2 < \ldots < i_{r-1} = n$ such that, for each $t \in \{1,\ldots,r-1\}$, $\supp{a_1, \ldots, a_t }$ is equal to $\{1,\ldots,i_t\}$.
\item \label{diaginequiv3} Let $t \in \{1,\ldots,r-3\}$. Then $\supp{a_t} \cap \supp{a_u} = \varnothing$ for all $u \in \{t+2,\ldots,r-1\}$.
\item \label{diaginequiv4} %
  The support of $a_2$ does not contain $1$.
\item \label{diaginequiv5} Let $t \in \{1,\ldots,r-1\}$. Then $i_t \in \supp{a_t} \cap \supp{a_{t+1}}$.
\item \label{diaginequiv6} Each integer in $\supp{a_r}$ lies in the support of a unique subspace in $C$.
\end{enumerate}
\end{lemma}

\begin{proof}
  We begin by fixing notation related to $a_r = \langle \sum_{\ell=1}^{n} \alpha_\ell e_\ell \rangle$ and $b_r = \langle \sum_{\ell=1}^{n} \beta_\ell e_\ell \rangle$. %
  Since $\A \s{r-1}{D} \B$, there exists an element in $D$ mapping $a_r$ to $b_r$, and so $\supp{b_r} = \supp{a_r}$. On the other hand, $\B \notin \A^D$, and so  $a_r \ne b_r$. Therefore, by scaling the basis vectors for $a_r$ and $b_r$, there exist $j,k \in \{1,\ldots,n\}$ such that $j < k$, $\alpha_j = \beta_j = 1$, and $\alpha_k$ and $\beta_k$ are distinct and nonzero. Reordering $\{e_1, \ldots, e_n\}$ if necessary, we may assume that $j = 1$. Then each element of $D$ that maps $a_r$ to $b_r$ also maps $\langle e_1 + \alpha_k e_k \rangle$ to $\langle e_1 + \beta_k e_k \rangle$; we will use this fact throughout the proof.

\begin{enumerate}
\item[\ref{diaginequiv2}] We show first that there is no partition of $C$ into  proper subsets $C'$ and $C''$ such that $\supp{C'} \cap \supp{C'' } = \varnothing$, so suppose otherwise, for a contradiction. Then, as $|C'| < r-1$ and $\A \s{r-1}{D} \B$, there exists an  $f \in D_{(C')}$ such that $a_r^f = b_r$. Multiplying $f$ by a scalar if necessary, we may assume that  $f_{1} = 1$. Then $f_{i} = \beta_i/\alpha_i$ for each $i \in \supp{a_r}$. Similarly, there exists $g \in D_{(C'')}$ with the same properties. As $\supp{ C' } \cap \supp{ C'' } = \varnothing$, there exists an $h \in D$ such that $h|_{\supp{ C' }} =f|_{\supp{ C' }}$ and $h|_{\supp{ C'' }} = g|_{\supp{{ C'' }}}$. Since
  $\supp{C} = \{1, \ldots, n\}$, we observe that
$h|_{\supp{a_r}} = f|_{\supp{a_r}} = g|_{\supp{a_r}}$. Hence $a_r^h = b_r$. Furthermore,
  by construction,
  $h \in D_{(C')} \cap D_{(C'')} = D_{C}$. Thus $\B \in \A^D$, a contradiction.

Next, by reordering $a_1, \ldots, a_{r-1}$ if necessary, we may assume that $1 \in \supp{a_1}$. 
Then by reordering $\{e_2,\ldots,e_{n}\}$ if necessary, we may assume that $\supp{a_1}$ is equal to $\{1,2,\ldots,i_1\}$ for some  $i_1 \ge 2$, since $a_1 \in \oD$.  
Thus the result holds for $t = 1$. We will use induction to prove the result in general, and to show that, for all $s \in \{2,\ldots,r-1\}$,
\begin{equation}
\label{eq:nontrivintersection}
\text{ there exists } w \in \{1,\ldots,s-1\} \text{ such that } \supp{a_s} \cap \supp{a_w} \ne \varnothing.
\end{equation}

Let $t \in \{2,\ldots,r-1\}$, let  $U_{t-1}:=\{a_1,\ldots,a_{t-1}\}$ and assume inductively that $\supp{U_{t-1}} = \{1,2,\ldots,i_{t-1}\}$. If $t \ge 3$ assume also that \eqref{eq:nontrivintersection} holds for all $s \in \{2,\ldots,t-1\}$. Since $C$ cannot be partitioned into two parts whose support has trivial intersection, 
$\supp{a_1, \ldots, a_{t-1}} \cap \supp{a_t, \ldots, a_{r-1}} \neq \varnothing$, 
so we may reorder $\{a_t,\ldots,a_{r-1}\}$ so that \eqref{eq:nontrivintersection} holds when $s = t$.

Suppose for a contradiction that $\supp{a_t} \subseteq \supp{U_{t-1}}$. Then \eqref{eq:nontrivintersection} (applied to each $s \in \{2,\ldots,t-1\}$) and Lemma \ref{lem:submats} imply that $D_{(C)}$ is equal to $D_{(C \setminus a_t)}$. Since $\A \s{r-1}{D} \B$, the latter stabiliser contains an element mapping $a_r$ to $b_r$. Hence the same is true for $D_{(C)}$, contradicting the fact that $\B \notin \A^D$. Therefore, %
we can reorder $\{e_{i_{t-1}+1},\ldots,e_{n}\}$ so that $\supp{a_t}$ contains $\{i_{t-1}+1,\ldots,i_{t}\}$ for some $i_{t} > i_{t-1}$, and
the result and \eqref{eq:nontrivintersection}  follow  by induction. Note in particular that $i_{r-1} = n$, since $\supp{C} = \{1, \ldots, n\}$. 

\item[\ref{diaginequiv3}] Let $m \in \{1,\ldots,r-1\}$ be such that
  $\supp{a_{m}}$ contains the integer $k$ from the first paragraph of
  this proof, and let $\mathcal{I}:=\{1,\ldots,m\}$. Then, using
  \eqref{eq:nontrivintersection} (for each $s \in \mathcal{I}
  \setminus \{1\}$) and Lemma~\ref{lem:submats}\ref{submats1}, we
  observe that every $g \in D_{(a_1,\ldots,a_m)}$ satisfies  $g_{1} =
  g_{k}$. Therefore, $a_r^g \neq b_r$ for all $g \in D_{(a_1, \ldots,
  a_m)}$. As $\A \s{r-1}{D} \B$, we deduce that $m = r-1$. In particular, $a_{r-1}$ is the unique subspace in $C$ whose support contains $k$. %
Swapping $e_k$ and $e_n$ if necessary, we may assume that $k = n$.

Now, for a contradiction, suppose that $\supp{a_t} \cap \supp{a_u} \ne
\varnothing$ for some $t \in
\{1,\ldots,r-3\}$ and  $u \in \{t+2,\ldots,r-1\}$, and assume that $u$ is the largest integer with
this property. Then \eqref{eq:nontrivintersection} and the maximality
of $u$ imply that $\supp{a_s} \cap \supp{a_{s-1}} \ne \varnothing$ for
all $s \in \{u+1,\ldots,r-1\}$. It now follows from Lemma
\ref{lem:submats}\ref{submats1}, together with a further application
of \eqref{eq:nontrivintersection} to each $s \in \{2,\ldots,t\}$, that
every $g \in E:= D_{(a_1,\ldots,a_t,a_u,\ldots,a_{r-1})}$ satisfies
$g_{1} = g_{n}$. Therefore, $a_r^g \neq b_r$ for all $g \in E$. However, $|(a_1,\ldots,a_t,a_u,\ldots,a_{r-1})| < r-1$, contradicting the fact that $\A \s{r-1}{D} \B$. 

\item[\ref{diaginequiv4}--\ref{diaginequiv5}] As in the proof of
  \ref{diaginequiv3}, we may assume that $k = n$. We observe from
  \ref{diaginequiv3} and \eqref{eq:nontrivintersection} that
  $\supp{a_t} \cap \supp{a_{t+1}} \ne \varnothing$ for all $t \le
  r-2$. Hence if 
  $1 \in \supp{a_2}$ then Lemma~\ref{lem:submats}\ref{submats1} shows
  that every $g \in D_{(a_2,\ldots,a_{r-1})}$ satisfies $g_{1} =
  g_{n}$ (since $k = n$).  This contradicts the fact that $\A
  \s{r-1}{D}\B$, and so \ref{diaginequiv4}  holds. Finally, since $\supp{a_t} \cap \supp{a_{t+1}} \ne \varnothing$ for each $t \le r-2$, we obtain \ref{diaginequiv5} by defining $i_0:=1$ and reordering the vectors in $\{e_{i_{t-1}+1},\ldots,e_{i_{t}}\}$ if necessary. In particular, for $t = r-1$, the assumption that $i_{r-1} = n = k \in \supp{a_r}$ gives the result.

\item[\ref{diaginequiv6}] %
  Suppose for a contradiction that some  $\ell \in \supp{a_r}$ lies in
  the support of more than one subspace in $C$. If $r = 3$, then $\ell
  \in {\supp{a_1} \cap \supp{a_2}}$ and we define $t := 2$. If instead
  $r > 3$, then \ref{diaginequiv3} implies that $\ell \in \supp{a_t}$
  for at least one $t \in \{2,\ldots,r-2\}$. In either case, $\ell \ne 1$, since $1 \not\in \supp{a_2}$ by \ref{diaginequiv4}, and $1 \not\in \supp{a_u} $ for $u \in \{3, \ldots, r-2\}$ by \ref{diaginequiv2}--\ref{diaginequiv3}. Furthermore, \ref{diaginequiv2} shows that $\ell \ne i_{r-1} = n$.

  Suppose first that $\alpha_{\ell} = \beta_{\ell} \; (\neq 0)$.
 Since the supports of $a_t, a_{t+1}, \ldots, a_{r-1}$ consecutively overlap,  Lemma~\ref{lem:submats}\ref{submats1} shows that $g_{\ell} = g_{n}$ for each $g \in D_{(a_t,\ldots,a_{r-1})}$. Since $\alpha_n \neq \beta_n$, no such $g$ maps $a_r$ to $b_r$, contradicting the fact that $\A \s{r-1}{D}
  \B$. Hence $\alpha_\ell \ne \beta_\ell$. However, each $g \in D_{a_1}$ satisfies $g_{1} = g_{\ell}$ if $r = 3$, as does
  each $g \in D_{(a_1,\ldots,a_{t})}$ if $r > 3$. Again, no such
  matrix $g$ maps $a_r$ to $b_r$, a contradiction. \qedhere
\end{enumerate}
\end{proof}

Recall that $\GGL$ denotes $\GL_n(\F)$, with $n,|\F| \ge 3$. Our next result is a key ingredient in the proof that
$\RC(\GGL,\Omega)$ is at most $n+2$. %

\begin{lemma}
\label{lem:subsetinequiv}
Let $r \ge 2$, and let $\A,\B \in \oD^{\,r}$ be such that $\A \s{r-1}{D} \B$ and $\B \notin \A^D$. Then there exists a subset $\Gamma$ of $\Delta$ of size $n+2-r$ such that $\B \notin \A^{\GGL_{(\Gamma)}}$.
\end{lemma}

\begin{proof}
If $r = 2$, then set $\Gamma = \Delta$. Since $\GGL_{(\Gamma)} = D$ and $\B \notin \A^D$, we are done. %
Assume therefore that $r \ge 3$. We will suppose for a contradiction that $n$ is the smallest dimension for which the present lemma does not hold, for this value of $r$. 
Since $\A \s{r-1}{D} \B$, we may also assume that $(a_1,\ldots,a_{r-1}) = (b_1,\ldots,b_{r-1})$. %
Let $C = \{a_1, \ldots, a_{r-1}\}$. As $\B \notin \A^D$, no element of $D_{(C)}$ maps $a_r$ to $b_r$. Therefore, $\B \notin \A^{\GGL_{(\Gamma)}}$ for a given subset $\Gamma$ of $\Delta$ if and only if no element of $\GGL_{(\Gamma \cup C)}$ maps $a_r$ to $b_r$. We split the remainder of the proof into two cases, depending on whether or not $|\supp{C}| = n$.

\medskip

\noindent \textbf{Case $\mathbf{|\supp{C}| < n}$:} 
Let $\Delta_C:=\{\langle e_j \rangle \mid j \in \supp{C}\}$, let $L$ be the subspace $\langle \Delta_C \rangle$ of $V$, and let $a_\ell$ and $b_\ell$ be the projections onto $L$ of $a_r$ and $b_r$, respectively. Lemma \ref{lem:submats}\ref{submats2} shows that %
  the diagonal entries corresponding to $\{1, \ldots,n\} \setminus \supp{C}$ of elements of $D_{(C)}$ %
  can take any multiset of nonzero values. Since no element of $D_{(C)}$ maps $a_r$ to $b_r$, it follows that there is no matrix in $D_{(C)}$ whose restriction to $L$ maps $a_\ell$ to $b_\ell$. %
  By the minimality of $n$, %
  there exists a subset $\Gamma_C$ of $\Delta_C$ of size $|\Delta_C|+2-r$ such that no element of $\GL(L)_{(\Gamma_C \cup C)}$ maps $a_\ell$ to $b_\ell$. Setting $\Gamma$ to be  $\Gamma_C \cup (\Delta \setminus \Delta_C)$, %
so that $|\Gamma| = n+2-r$, we observe that no element of $\GGL_{(\Gamma \cup C)}$ maps $a_r$ to $b_r$. This is a contradiction, and so the lemma follows %
in this case.

\medskip

\noindent \textbf{Case $\mathbf{|\supp{C}| = n}$:} In this case, Lemma~\ref{lem:diaginequiv} applies,
so with the notation of that lemma, let $$\Gamma:=\Delta \setminus \{\langle e_{i_1} \rangle, \ldots, \langle e_{i_{r-2}} \rangle\}.$$ Then $|\Gamma| = n+2-r$ and $\langle e_1 \rangle, \langle e_{n} \rangle \in \Gamma$, since $i_1 \ge 2$ and $i_{r-1} = n$.

Let $g \in \GGL_{(\Gamma \cup C)}$. To complete the proof, we will show 
that $a_r^g = a_r \ne b_r$, by  showing that $g|_{\supp{a_r}}$ is scalar.  We will first show that $g$ is lower triangular. %
It is clear that $g$ stabilises $\langle e_1 \rangle \in \Gamma$. Suppose inductively that $g$ stabilises $\langle e_1, e_2, \ldots, e_s \rangle$ for some $s \in \{1,\ldots,n-1\}$. If $\langle e_{s+1} \rangle \in \Gamma$,
then $g$ stabilises $E_{s+1}:= \langle e_1, e_2, \ldots, e_s \rangle + \langle e_{s+1} \rangle = \langle e_1, e_2, \ldots, e_{s+1} \rangle$. Otherwise, $s+1 = i_t$ for some $t \in \{1,\ldots,r-2\}$, and then Lemma \ref{lem:diaginequiv}\ref{diaginequiv2} shows that $\{s+1\} \subsetneq \supp{a_t} \subseteq \{1, \ldots, s+1\}$.
In this case, $g$ again stabilises $\langle e_1, e_2, \ldots, e_s \rangle + a_t = E_{s+1}$. %
Hence by induction, $g$ is lower triangular.

Now, let $\mathcal{I}:=\{i_1,\ldots,i_{r-1}\}$, let $\mathcal{U}$ be the set of integers that each lie in the support of a unique subspace in $C$, and let $\mathcal{J}:=\mathcal{I} \cup \mathcal{U}$. %
We will show next that 
 $g|_{\mathcal{J}}$ is diagonal, by fixing $j \in \mathcal{J}$ and proving that $g_{kj} = 0$ whenever $k > j$. First, if $\langle e_k \rangle \in \Gamma$, then $g_{kj} = 0$, so $g_{nj} = 0$. Hence we may also assume that $k \in \mathcal{I} \setminus \{i_{r-1}\}$.

Suppose %
inductively that $g_{i_u,j}=0$ for some $u \ge 2$ (the base case here is $u = r-1$, so that $i_u = n$). We will show that if $i_{u-1} > j$, then $g_{i_{u-1},j} = 0$. By Lemma~\ref{lem:diaginequiv}\ref{diaginequiv5}, $i_{u-1}, i_u \in \supp{a_u}$ and furthermore Lemma~\ref{lem:diaginequiv}\ref{diaginequiv2}--\ref{diaginequiv3} shows that $\supp{a_u} \cap \mathcal{I} = \{i_{u-1},i_u\}$.  Thus by the previous paragraph and our inductive assumption, $g_{kj} = 0$ for all $k \in \supp{a_u} \setminus \{j,i_{u-1}\}$. In fact, Lemma \ref{lem:diaginequiv}\ref{diaginequiv2}--\ref{diaginequiv3} shows that each integer in $\supp{a_u}$ less than $i_{u-1}$ lies in $\supp{a_{u-1}}$. As $i_{u-1} > j \in \mathcal{J}$, we deduce from the definition of $\mathcal{J}$ that $j \notin \supp{a_u}$. Thus $g_{kj} = 0$ for all $k \in \supp{a_u} \setminus \{i_{u-1}\}$. As $g$ stabilises $a_u$, we deduce that $g_{i_{u-1},j} = 0$. Therefore, by induction, $g_{kj} = 0$ for all $k \ne j$ and so $g|_{\mathcal{J}}$ is diagonal.

Finally, we will show that $g|_{\mathcal{J}}$ is scalar. Let $j,k\in \mathcal{J} \cap \supp{a_t}$ for some $t \in \{1,\ldots,r-1\}$. As $g$ stabilises $a_t$, and as
$g|_{\mathcal{J}}$ is diagonal,  we deduce that
\begin{equation}
\label{eq:gab}
g_{jj} = g_{kk}.
\end{equation}
Now, by Lemma \ref{lem:diaginequiv}\ref{diaginequiv5}, $i_t \in \supp{a_t} \cap \supp{a_{t+1}}$ for each $t \in \{1,\ldots,r-2\}$, so $i_t \in \mathcal{J} \cap \supp{a_t} \cap \supp{a_{t+1}}$. Thus starting from $t = 1$ and proceeding by induction on $t$, it follows from \eqref{eq:gab} that $g_{jj} = g_{kk}$ for all $j, k  \in \mathcal{J}$, i.e.~$g|_{\mathcal{J}}$ is a scalar. Since $\supp{a_r} \subseteq \mathcal{J}$ by Lemma~\ref{lem:diaginequiv}\ref{diaginequiv6}, we deduce that $a_r^g = a_r \ne b_r$, as required.
\end{proof}

The following lemma is strengthening of
  Lemma~\ref{lem:subsetinequiv} in the case $|\F| = 3$ and $r = 2$,
  in which  the subset $\Gamma$ now has size $n + 1-r = n-1$.

\begin{lemma}
\label{lem:subsetinequivq3}
Suppose that $|\F| = 3$, and let $\A,\B \in \oD^{\,2}$. Suppose also that $\A \s{1}{D} \B$ and $\B \notin \A^D$. Then there exists a subset $\Gamma$ of $\Delta$ of size $n-1$ such that $\B \notin \A^{\GGL_{(\Gamma)}}$.
\end{lemma}

\begin{proof}
Since $ \A \s{1}{D} \B$, without loss of generality $a_1 = b_1$,
and there exists an element of $D$ mapping $a_2$ to $b_2$.
  Hence $a_2$ and $b_2$ have equal supports. Reordering the basis for $V$ if necessary, we may also assume that $\supp{a_1} = \{1,2,\ldots,m\}$ for some $m \ge 2$. Then by Lemma \ref{lem:submats}, the upper-left $m \times m$ submatrix of each matrix in $D_{a_1}$ is a scalar, while the remaining diagonal entries can be chosen independently. As $\B \notin \A^D$, no matrix in $D_{a_1}$ maps $a_2$ to $b_2$. We may therefore assume (by reordering the basis vectors in $\{e_1,\ldots,e_m\}$ and/or swapping $\A$ and $\B$ if necessary) that the projections of $a_2$ and $b_2$ onto $\langle e_1, e_2 \rangle$ are $\langle e_1 + e_2 \rangle$ and $\langle e_1 - e_2 \rangle$, respectively.

Now, let $\Gamma:=\Delta \setminus \{\langle e_2 \rangle\}$, let $g
\in \GGL_{(\Gamma \cup \{a_1\})}$,  and notice that $g$ is diagonal outside of the second row.  Write $a_1$ as $\langle
  \sum_{i=1}^m \alpha_i e_i \rangle$, with $\alpha_1 = 1$ and
  $\alpha_i \neq 0$ for all $i \in \{2,\ldots,m\}$.
  Since $a_1^g = a_1$ we deduce that without loss of generality the top left $2 \times 2$ submatrix of $g$ is $$\left( \begin{array}{cc} 1 & 0 \\
             g_{21} & 1 + \alpha_2 g_{21} \end{array} \right).$$ Let
         $v$ be the projection of $(e_1 + e_2)^g$
       onto $\langle e_1, e_2 \rangle$. 
     Recall that $\alpha_2 \ne 0$, and note that $g_{22} \neq 0$, since $g$ is invertible. Hence if $g_{21} = 1$, then $\alpha_2 = 1$ and $v = -e_1-e_2$; if $g_{21} = -1$, then $\alpha_2 = -1$ and $v = -e_2$; and if $g_{21} = 0$, then $v = e_1+e_2$. Hence, in each case, $v$ does not span $\langle e_1 - e_2 \rangle = b_2|_{\langle e_1, e_2 \rangle}$. Therefore $a_2^g \ne b_2$, and hence $\B \not\in \A^{G_{(\Gamma)}}$.
\end{proof}

Although the next result holds for all $\F$, it will only be useful in the case $|\F| = 3$.

\begin{proposition}
\label{prop:nequivalent}
Let $\X,\Y \in \Omega^{n+1}$ such that $\X \s{n}{\GGL} \Y$, and suppose that $\langle \X \rangle = V$. Then $\Y \in \X^\GGL$.
\end{proposition}

\begin{proof}
  As $\dim(\langle \X \rangle) = n$, we may assume by Lemma \ref{lem:tuplefirstsubs-norm} that $x_i = y_i = \langle e_i \rangle$ for $i \in \{1, \ldots, n\}$. Let $S := \supp{x_{n+1}}$ and $T := \supp{y_{n+1}}$. 
  We will show that $S = T$; %
  it will then follow that there exists an element of $D = \GGL_{(\Delta)}$ mapping $\x_{n+1}$ to $\y_{n+1}$, and so $\Y \in \X^\GGL$.
  
If $S = \{1, \ldots, n\} = T$, then we are done. Otherwise,  exchanging $\X$ and $\Y$ if necessary (note that $\langle \Y \rangle = V$), we may assume that there exists an element $t \in \{1,\ldots,n\} \setminus S$. Let $\Gamma:=\Delta \setminus \{\langle e_t \rangle\}$. Then since $\X \s{n}{\GGL} \Y$, there exists an element of $\GGL_{(\Gamma)}$ mapping $\x_{n+1}$ to $\y_{n+1}$. As $\GGL_{(\Gamma)}$ stabilises each subspace $\langle e_i \rangle$ with $i \in S$, it follows that $S = T$, as required.
\end{proof}

We are now able to prove this section's main theorem. 

\begin{theorem}
  \label{thm:relupper}
Suppose that $n$ and $|\F|$ are at least $3$. Then $\RC(\GL_{n}(\F), \Omega)$ is at most $n+2$.
Moreover, %
$\RC(\GL_{n}(3),\Omega) \le n$.
\end{theorem}

\begin{proof}
  Let $k \in \{n,n+1,n+2\}$, with $k =n+2$ if $|\F| > 3$. Additionally, let $\X,\Y \in \Omega^u$ with $u > k$ and $\X \s{k}{\GGL} \Y$, where $\GGL = \GL_{n}(\F)$. It suffices to prove that $\Y \in \X^\GGL$. Suppose, for a contradiction, that $n$ is the minimal dimension for which the theorem does not hold (for a fixed $\F$), and that $\Y \notin \X^\GGL$. Then for each $m \in \{2,\ldots,n-1\}$, using Proposition~\ref{prop:lodabounds}\ref{proppart:scott} in the case $m = 2$, we obtain $\RC(\GL_{m}(\F),\PG_1(\F^m)) < k$. Since $\Y \notin \X^\GGL$, Lemma \ref{lem:tuplefulldim} yields $\langle \X \rangle = V$. Hence by Lemma \ref{lem:tuplefirstsubs-norm}, we may assume without loss of generality\footnote{If the basis vectors for $V$ are reordered, as required by several of this section's earlier proofs, then we can reorder the subspaces in $(\x_1,\ldots,\x_{n})$ and $(\y_1,\ldots,\y_{n})$ in the same way to preserve this equality.} that
  $x_i = y_i = \langle e_i \rangle$ for $i \in \{1, \ldots, n\}$, and furthermore
  that all subspaces in $\X$ are distinct, so that $\x_i,\y_i \in \oD$ for each $i \ge n+1$.

  We will first consider the case $k \ge n+1$. Since $\X \s{n+1}{\GGL} \Y$, Lemma~\ref{lem:eqsupps} yields $\supp{\x_i} = \supp{\y_i}$ for all $i$. However, $\Y \notin \X^\GGL$. Hence there exists an integer $r \ge 2$ and subtuples $\A$ of $\X$ and $\B$ of $\Y$, with $\A,\B \in \oD^{\, r}$, such that
$(x_1, \ldots, x_n, a_1, \ldots, a_r)$ and $(x_1, \ldots, x_n, b_1, \ldots, b_r)$ are $(n+r-1)$-equivalent, but not equivalent, under $G$. 
  Equivalently, $\A \s{r-1}{D} \B$ and $\B \notin \A^D$.

  If $k = n+2$, then by Lemma~\ref{lem:subsetinequiv}, there exists a set $\Gamma := \{\langle e_{i_1} \rangle, \ldots,  \langle e_{i_{k-r}} \rangle \}$ such that $\B \notin \A^{\GGL_{(\Gamma)}}$. However, this means that the subtuples $(x_{i_1}, \ldots, x_{i_{k-r}}, a_1, \ldots, a_r)$ and $(x_{i_1}, \ldots, x_{i_{k-r}}, b_1, \ldots, b_r)$
  are not equivalent under $\GGL$. This contradicts the assumption that $\X \s{k}{\GGL} \Y$. Hence in this case, $\Y \in \X^{\GGL}$, as required, so $\RC(\GGL) \le n+2$.
  If $|\F| > 3$, then we are done.

Therefore, assume for the rest of the proof that $|\F| = 3$ and suppose first that $k = n+1$. By the previous paragraph,
 $\RC(\GGL) \le n+2$.
 Therefore, to prove that  $\RC(\GGL) \le k$, it suffices to show that $\X \s{n+2}{\GGL} \Y$ whenever $\X \s{k}{\GGL} \Y$. Thus by replacing $\X$ and $\Y$ by suitable subtuples, if necessary, we may assume that $u = n+2$.
 In this case, $r = 2$, and by Lemma \ref{lem:subsetinequivq3}, there exists a subset $\Gamma$ of $\Delta$ of size $k-r$ such that $\B \notin \A^{\GGL_{(\Gamma)}}$. Arguing as in the previous paragraph, this contradicts the assumption that $\X \s{k}{\GGL} \Y$. Thus $\RC(\GGL) \le n+1$. 

Finally, suppose that  $k = n$. Since $\RC(\GGL) \le n+1$, we may assume that $u = n+1$. However, since $\X \s{n}{\GGL} \Y$ and $\langle \X \rangle = V$, Proposition \ref{prop:nequivalent} shows that $\Y \in \X^\GGL$. Therefore, $\RC(\GGL) \le n$.
\end{proof}

\section{Action on 1-spaces: lower bounds}\label{sec:lower1}

In this section, we again assume that $|\F| \ge
  3$, and write $\Omega:=\Omega_1 = \PG_1(V)$. We drop the assumption that $n \ge 3$ and permit $n = 2$. We shall now prove lower bounds for the relational complexity of each group $\gp$ satisfying ${\SL_n(\F) \trianglelefteq \gp \le \GamL_n(\F)}$, acting on $\Omega$.

For some results in this section, we will assume that $\F =
  \F_q$ is finite, and when doing so we fix a primitive element $\omega$, and
  assume that $q = p^f$ for $p$ prime. Additionally,  we 
will write $\PGamL_n(q)/\PSL_n(q) = \langle \delta, \phi \rangle$,
with $\PGL_n(q)/\PSL_n(q) = \langle \delta \rangle$.
Here, the
automorphism $\phi$ can be chosen to be
induced by the automorphism of $\GL_n(q)$ which
raises each matrix entry to its $p$th power, and
with a slight abuse of notation,
we will also write $\phi$ to denote this automorphism of $\GL_n(q)$,
and to denote a generator for $\Aut(\F_q)$. %
If $\F$ is an arbitrary field, then the group $\GamL_n(\F)$ is
  still a semi-direct product of $\GL_n(\F)$ by
  $\Aut(\F)$ (see, for example, \cite[Theorem 9.36]{rotman}), but of course $\GL_n(\F)/\SL_n(\F)$ and $\Aut(\F)$ need
  not be cyclic.

We
let $Z := Z(\GL_n(\F))$, and will write $\id_{n}$ for the $n \times n$ identity matrix,
and $E_{ij}$ for the $n \times n$ matrix with $1$ in the
$(i,j)\th$ position and $0$ elsewhere. We write
$A \oplus B$ for the block diagonal matrix with blocks $A$ and $B$.

Our first result is completely general and easy to prove,
  although we shall later prove much tighter bounds for various
  special cases.

\begin{theorem}\label{thm:lower_bound_n}
 Let $\F$ be arbitrary, and let $H$ satisfy $\SL_n(\F) \trianglelefteq H
 \le \GamL_n(\F)$. Then $\RC(\gp, \Omega) \ge n$.
\end{theorem}

\begin{proof}%
  Define $\X, \Y \in \Omega^{n}$ by $x_i = y_i = \langle e_i \rangle$ for $i \in \{1, \ldots, n-1\}$, with $x_n = \langle \sum_{i = 1}^n e_i \rangle$ and $y_n = \langle \sum_{i = 1}^{n-1} e_i \rangle$.
  Then %
$\dim(\langle X \rangle) = n$ and $\dim(\langle Y \rangle) = n-1$, so
no element of $\GamL_{n}(\F)$ maps $\X$ to $\Y$. Hence $\Y \not\in \X^{H}$. 

Now, let $h_{\ell}:=\id_{n}-E_{\ell n}$ for each $\ell \in \{1, \ldots, n-1\}$, and $h_{n}:=\id_{n}$. Then $h_\ell \in \SL_n(\F) \le H$ and $(\X \setminus \x_\ell)^{h_\ell}=(\Y \setminus \y_\ell)$, for each $\ell \in \{1, \ldots, n\}$.
Therefore $\X \s{n-1}{H} \Y$, and so the result follows.
\end{proof}

Our next two results focus on the special cases $n = 2$ and $n= 3$.

\begin{lemma}\label{lem:PGamL2(q)}
  Assume that %
    $q \ge 8$, and let $\gp$
  satisfy $\SL_2(q) \trianglelefteq \gp \le \GamL_2(q)$.
Then $\RC(\gp) \geq 4$, except that $\RC(\SigL_2(9)) = 3$.
\end{lemma} 

\begin{proof}
 The claim about $\SigL_2(9)$ is an easy computation in GAP using \cite{RComp}, so exclude this group from now on.
  We divide the proof into two cases. For each, we define $\X, \Y \in \Omega^4$ such that $\X \s{3}{\gp} \Y$ but $\Y \not\in \X^{\gp}$. In both cases, we set
  $(\X \setminus x_4) = (\Y \setminus y_4) = (\langle e_1 \rangle, \langle e_2 \rangle, \langle e_1 + e_2 \rangle)$.

\medskip

\noindent \textbf{Case (a)}: Either $q$ is even, or $\gp \not\le \langle Z, \SigL_2(q)\rangle$, where $Z = Z(\GL_n(\F))$. If $q$ is odd, then let $\alpha \in \F_p^\ast \setminus \{1\}$, and otherwise let $\alpha = \omega^3$, so that $\alpha$ is not in the orbit $\omega^{\langle \phi \rangle}$.
Then let $x_4 = \langle e_1+\omega e_2 \rangle$ and $y_4 = \langle e_1+\alpha e_2 \rangle$.

The stabiliser in $\gp$ of $(\X \setminus x_4) = (\Y \setminus y_4)$ is contained in $\langle Z,  \phi \rangle$.
As $\alpha \not\in \omega^{\langle \phi \rangle}$, no element of this stabiliser
maps $x_4$ to $y_4$, and so $\Y \not\in \X^{\gp}$.
On the other hand, for each $j \in \{1,2,3,4\}$, the matrix $g_j \in
\GL_2(q)$ given below maps $(\X \setminus \x_{j})$ to $(\Y \setminus
\y_{j})$. 
$$\begin{array}{c}
g_1 = \begin{pmatrix} 1 & (\alpha-\omega)(1-\omega)^{-1} \\ 0 & 1-(\alpha-\omega)(1-\omega)^{-1} \end{pmatrix}, \quad
g_2 = \begin{pmatrix} 1-(\omega\alpha^{-1}-1)(\omega-1)^{-1} & 0 \\  (\omega\alpha^{-1}-1)(\omega-1)^{-1} & 1 \end{pmatrix}, \\
g_3 = \begin{pmatrix} 1 & 0 \\ 0 & \alpha\omega^{-1}  \end{pmatrix}, \quad 
g_4 = I_2.
  \end{array}$$
  
If $q$ is even, then some scalar multiple of $g_j$ lies in $H$ for all $j$, so $\X \s{3}{\gp} \Y$ and
we are done.  If instead $q$ is odd, then
our assumption that $\gp \not\le \langle Z, \SigL_2(q)\rangle$ implies that $\gp$ contains a scalar multiple of an element  
  $\diag(\omega, 1)\phi^i$ for some $i \ge 0$, as $\diag(\omega,1)$ induces the automorphism $\delta$ of $\PSL_2(q)$. Hence for each $j$, there exists $\phi^{i_j} \in \Aut(\F_q)$ such that a scalar multiple of $g_j \phi^{i_j}$ lies in $H$. Since $\alpha \in \F_p^\ast$, each $\phi^{i_j}$ fixes $\Y$, and thus $\X \s{3}{\gp} \Y$.

\medskip

\noindent \textbf{Case (b)}: $q$ is odd and $\gp \le \langle Z, \SigL_2(q) \rangle$.  Since $\gp \ne \SigL_2(9)$, and since Proposition~\ref{prop:lodabounds}\ref{proppart:scott} yields the result when $\gp = \SL_2(9)$, we may assume that $q > 9$. We generalise Hudson's \cite[\S5.4]{Hudson} proof that $\RC(\SL_2(q),\Omega) \geq 4$. First, let $\mathcal{S} := \F_q \setminus \{0, 1, -1\}$ and $\mathcal{T} := \F_q \setminus \{0, 1\}$, and for each $\lambda \in \mathcal{S}$ define a map $\theta_\lambda: \mathcal{T} \to \F_q$ by $\mu \mapsto (1-\lambda^2 \mu)(1-\mu)^{-1}$. We will show that there exist elements $\lambda \in \mathcal{S}$ and $\tau \in \mathcal{T}$ satisfying the following conditions: 
\begin{center}
(i) $(\tau)\theta_\lambda$ is a square in $\F_q^*$, and
$\quad$ (ii)  no automorphism of $\F_q$ maps $\tau$ to $\lambda^2\tau$.
\end{center}
It is easy to see that for each
$\lambda \in \mathcal{S}$, the image $\mathrm{im}(\theta_\lambda) = \F_q
\setminus \{1,\lambda^2\}$, so the map $\theta_\lambda$ is injective, and the preimage of any nonzero square in $\mathrm{im}(\theta_\lambda)$ lies in $\mathcal{T}$ and satisfies Condition (i). Hence for each $\lambda \in \mathcal{S}$, there are precisely $(q-1)/2 - 2$ choices for $\tau \in \mathcal{T}$ satisfying Condition (i).

Given $\lambda \in \mathcal{S}$,  since $\lambda^2 \neq 1$, Condition (ii) is equivalent to requiring that $\lambda^2 \tau \ne \tau^{p^k}$ for all $k \in \{1,\ldots,f-1\}$, i.e.~$\lambda^2 \ne \tau^{p^k-1}$ for all $k$. There are exactly $(q-3)/2 = (q-1)/2 - 1$ distinct squares of elements of 
  $\mathcal{S}$,
and precisely $(q-1)/(p-1)$ elements in $\F_q^*$ that are $(p-1)$-th powers. Hence if $p > 3$, then there exists $\lambda \in \mathcal{S}$ such that $\lambda^2$ is not a $(p-1)$-th power in $\F_q$. Observe that then $\lambda^2$ %
is not a $(p^k-1)$-th power for any $k$, and so this $\lambda$ and any corresponding $\tau$ from the previous paragraph satisfy both conditions.%

Suppose instead that $p = 3$, and fix $\lambda \in \mathcal{S}$. The
number of elements $\tau \in \F_{3^f}^*$ not satisfying (ii),
i.e.~with $\lambda^2 = \tau^{3^k-1}$ for some $k \in
\{1,\ldots,f-1\}$, is at most \[(3-1)+(3^2-1)+\cdots+(3^{f-1}-1) =
  (3+3^2+\cdots+3^{f-1})-(f-1).\] On the other hand, we established
that the number of elements $\tau \in \mathcal{T}$ satisfying (i) is equal to \[(3^f-1)/2-2 = (3-1)(1+3+3^2+\cdots+3^{f-1})/2-2 = (3+3^2+\cdots+3^{f-1})-1.\] Since $q > 9$, and hence $f > 2$, there again exists $\tau\in \mathcal{T}$ satisfying both conditions.

Finally, fix such a $\lambda \in \mathcal{S}$ and $\tau \in \mathcal{T}$, and complete the  definition of $\X, \Y \in \Omega^{4}$ by setting $x_4 = \langle e_1+ \tau e_2 \rangle$  and $y_4 = \langle e_1+ \lambda^2 \tau e_2 \rangle$.
The stabiliser in $H$ of $(\X \setminus x_4) = (\Y \setminus y_4)$ is contained in $\langle Z, \phi \rangle$.
By Condition (ii), no such element maps $x_4$ to $y_4$, so $\Y \notin
\X^\gp$. However, the proof of \cite[Theorem 5.4.6]{Hudson}
uses Condition (i) to exhibit explicit elements of $\SL_2(q)$ mapping each $3$-tuple of $\X$ to the corresponding $3$-tuple of $\Y$. Therefore, $\X \s{3}{\gp} \Y$, and the result follows.
\end{proof}

\begin{lemma}\label{lem:PSL-n=3}
  Assume that $\PSL_3(\F) \neq \PGL_3(\F)$, and let $\gp$ satisfy $\SL_3(\F) \trianglelefteq \gp \le \GamL_3(\F)$. If $\F$ is finite, or if $\gp \le \GL_3(\F)$, then $\RC(\gp) \geq 5$.
\end{lemma}

\begin{proof}
If $|\F| = 4$, then we verify the result in GAP using \cite{RComp}, so assume that $|\F| \ge 7$. If $\F$ is finite, then let $\lambda := \omega$, whilst if $\F$ is infinite, then let $\lambda$ be any element of $\F^\ast$ of multiplicative order at least $3$.
  Define $\X,\Y \in \Omega^{5}$ by $x_i = y_i = \langle e_i \rangle$
  for $i \in \{1, 2, 3\}$,  $x_4 = y_4 = \langle e_1 + e_2 + e_3
  \rangle$, $x_5 = \langle e_1+\lambda e_2+\lambda^2 e_3 \rangle$, and
  $y_5 = \langle  e_1+\lambda^{-1} e_2+ \lambda^{-2} e_3 \rangle$,
  so that $x_5 \neq y_5$. 
  
  We first show that $\Y \not\in \X^{H}$.
  The stabiliser in $H$ of $(\X \setminus x_5) = (\Y \setminus y_5)$
  lies in  $H \cap \langle Z, \Aut(\F) \rangle$, so if $\F$ is
  infinite then we are done. Assume therefore that $\F =
    \F_q$. %
    If $%
    x_5^{\phi^i}= \y_5$, then $\lambda^{p^i} = \lambda^{-1}
= \lambda^{p^f - 1}$. Since $i \in  \{0, \ldots, f-1\}$ and
  $\lambda = \omega$, we deduce that $(p, f, i) \in
\{(2,2,1),(3,1,0)\}$, contradicting $q \ge 7$. 
Thus $\Y \not\in \X^{\gp}$.

Next, for all $\F$, we show that %
$\X \s{4}{\gp} \Y$. Let %
$$\begin{array}{c}
g_1:=\begin{pmatrix} \lambda & \lambda+1 &\lambda+\lambda^{-1}\\
0 & -1 & 0 \\
0 & 0 & - \lambda^{-1} \end{pmatrix}, \quad
        g_2:=\begin{pmatrix} -\lambda & 0 & 0\\
\lambda+1 & 1 & 1 + \lambda^{-1} \\
0 & 0 & - \lambda^{-1} \end{pmatrix}, \\
\\
 g_3:=\begin{pmatrix} -\lambda & 0 & 0\\
0 & -1 & 0 \\
\lambda+\lambda^{-1} & 1 + \lambda^{-1}  & \lambda^{-1} \end{pmatrix}, \quad g_4:=\begin{pmatrix} \lambda^{2} & 0 & 0\\
0 & 1 & 0 \\
0 & 0 &  \lambda^{-2} \end{pmatrix}, \ \mbox{ and } \ \  g_5 := \id_{3}.
  \end{array}$$
  Observe that $\det(g_\ell)=1$ for each $\ell \in \{1, \ldots, 5\}$, and so $g_\ell \in \SL_3(\F) \leq H$. It is also easy to check that $(\X \setminus \x_{\ell})^{g_\ell}=(\Y \setminus \y_{\ell})$ for each $\ell$. Thus $\X \s{4}{\gp}\Y$, and so $\RC(\gp) \geq 5$.
\end{proof}

Our remaining results hold for all sufficiently large $n$. The
  first  is specific to $\GL_n(\F)$.

\begin{proposition}\label{prop:PGLlow}
 If $n \geq 3$ and $|\F| \ge 4$, then $\RC(\GL_n(\F),\Omega) \geq n+2$.
\end{proposition}

\begin{proof}
  Since $|\F| \ge 4$, there exists an element $\lambda \in \F^\ast$ such that $\lambda \neq \lambda^{-1}$ (so $\lambda \neq -1$). Define $\X, \Y \in \Omega^{n+2}$ by $x_i = y_i = \langle e_i \rangle$ for $i \in \{1, \ldots, n\}$, $x_{n+1} = y_{n+1} = \langle \sum_{i = 1}^n e_i \rangle$, $x_{n+2} = \langle e_1+\lambda e_2 \rangle$ and $y_{n+2} = \langle  e_1+ \lambda^{-1} e_2 \rangle$.

   The stabiliser in $\GL_n(\F)$ of $(\X \setminus x_{n+2}) = (\Y \setminus y_{n+2})$ is the group of scalar matrices,  so $\Y \not\in \X^{\GL_n(\F)}$.
    Additionally, it is easily verified that, for each $j \in \{1, \ldots, n+2\}$, the matrix $g_j \in \GL_n(q)$ given below maps $(\X \setminus \x_{j})$ to $(\Y \setminus \y_{j})$.
    $$\begin{array}{c}
        g_1 = \begin{pmatrix} \lambda & 1 + \lambda \\
          0 & -1 \end{pmatrix} \oplus \lambda I_{n-2}, \quad g_2 =
              \begin{pmatrix} -1 & 0 \\ 1+\lambda^{-1} & \lambda^{-1} \end{pmatrix} \oplus \lambda^{-1}I_{n-2}, \quad  g_{n+1} = \diag(\lambda, \lambda^{-1}, \lambda, \ldots, \lambda),  \\
        g_j = g_{n+1} + (\lambda-\lambda^{-1})E_{j2}  \ \mbox{ for $j \in \{3, \ldots, n\}$}, \quad g_{n+2} = I_{n}.\end{array}$$
Hence $\X \s{n+1}{\GL_n(\F)} \Y$, and so the result follows.
\end{proof}

In the light of Proposition~\ref{prop:PGLlow}, the next result in
particular bounds the relational complexity of all remaining groups
when $\PSL_n(\F) = \PGL_n(\F)$.

\begin{lemma}\label{lem:PGamLn(q)}
 Let $\F$ be arbitrary, assume that $n \geq 3$, and let $\gp$ satisfy $\GL_n(\F) \trianglelefteq \gp \le \GamL_n(\F)$ and $\gp \neq
  \GL_n(\F)$.  Then $\RC(\gp) \geq n+3$.
\end{lemma}

\begin{proof} 
  Since $\GL_n(\F)$ is a proper subgroup of $\gp$,
there exists a nontrivial $\psi \in \gp \cap \Aut(\F)$ and an element $\lambda \in \F^*$ with $\lambda^\psi \ne \lambda$.
We define $\X, \Y \in \Omega^{n+3}$ by $x_i = y_i = \langle e_i \rangle$ for $i \in \{1, \ldots, n\}$,
$$x_{n+1} = y_{n+1} = \Big\langle \sum_{i=1}^{n} e_i \Big\rangle, \ x_{n+2} = y_{n+2} =  \langle e_1+e_2+\lambda e_3 \rangle, \  x_{n+3} = \langle e_1+\lambda e_2 \rangle, \ y_{n+3} = \langle  e_1+ \lambda^{\psi} e_2 \rangle.$$ 
We claim that $\X \s{n+2}{\gp} \Y$, but $\Y \not\in \X^{\gp}$, from which
the result will follow.

The stabiliser in $H$ of $(x_1, \ldots, x_{n+1}) = (y_1 ,\ldots, y_{n+1})$ is contained in $\langle Z, \Aut(\F) \rangle$. However, no element of $\langle Z, \Aut(\F) \rangle$ maps $(x_{n+2},x_{n+3})$ to $(y_{n+2},y_{n+3})$, so $\Y \not\in \X^{\gp}$.
The reader may verify that, for each $j \in \{1, \ldots, n+3\}$, the
element $h_j \in \langle \GL_n(\F), \psi \rangle \le H$ given below maps $(\X \setminus \x_{j})$ to $(\Y \setminus \y_{j})$, where we define $\tau := (\lambda - 1)^{-1}$ (notice that $\lambda \neq 1$). 
 $$\begin{array}{c}
   h_1 = \begin{pmatrix}1  & -\tau(\lambda^{\psi} - \lambda) \\
      0 & 1 + \tau(\lambda^{\psi} - \lambda)\end{pmatrix} \oplus I_{n-2}, \quad 
          h_2 = \begin{pmatrix} 1 - \tau(\lambda(\lambda^{-1})^\psi - 1) & 0 \\
            \tau (\lambda(\lambda^{-1})^\psi - 1) & 1 \end{pmatrix} \oplus I_{n-2}, \\
    h_3 = \left( \begin{pmatrix}
      1 - \tau (\lambda(\lambda^{-1})^{\psi^{-1}}-1) & 0 & 0\\
      0 & 1 - \tau (\lambda(\lambda^{-1})^{\psi^{-1}}-1) & 0 \\
      \tau (\lambda(\lambda^{-1})^{\psi^{-1}}-1)  & \tau (\lambda(\lambda^{-1})^{\psi^{-1}}-1)  & 1\end{pmatrix} \oplus I_{n-3} \right) \psi,\\
    h_j = \left( \diag (1, 1, \lambda^{-1}\lambda^{\psi^{-1}}, 1, \ldots, 1) + (1-\lambda^{-1}\lambda^{\psi^{-1}})E_{j3} \right) \psi \ \mbox{ for  $j \in \{4, \ldots, n\}$}, \\ h_{n+1} = \diag(1, 1, \lambda^{-1}\lambda^{\psi^{-1}}, 1, \ldots, 1)\psi, \quad
h_{n+2} = \psi, \quad h_{n+3} = I_n.
\end{array}$$
Hence $\X \s{n+2}{\gp} \Y$, and the result follows.
\end{proof}

\begin{lemma}\label{lem:general_n+2}
  Let $\F$ be arbitrary, assume that $n \ge 4$, and let $H$ satisfy $\SL_n(\F) \trianglelefteq H \le \GamL_n(\F)$ and $H \not\le \GL_n(\F)$. Then $\RC(H) \ge n+2$.
\end{lemma}

\begin{proof}
  Since $H \not\le \GL_n(\F)$, there exist elements $h \psi \in H$ and
  $\lambda \in \F^*$ such that $h \in \GL_n(q)$, $\psi \in
\Aut(\F_q)$, and $\lambda^\psi \ne \lambda$. Let $\X, \Y \in \Omega^{n+2}$ be as in the proof of Lemma~\ref{lem:PGamLn(q)}, but supported only on the first $n-1$ basis vectors, so that $\langle e_n \rangle$ lies in neither $\X$ nor $\Y$, and $x_n = y_n = \langle\sum_{i = 1}^{n-1} e_i\rangle$. Just as in that proof, one may check that $\Y \not\in X^{H}$, but $X \s{n+1}{\gp} \Y$.
  \end{proof}

 The next result applies, in particular, to all  groups $H$ such
    that $\SL_n(\F) \trianglelefteq H$ and either $H < \GL_n(\F)$ or $H
    \le \SigL_n(\F) \neq \GamL_n(\F)$. We write $\F^{\times n}$ for the subgroup of $\F^\ast$ consisting of $n$-th powers, which is the set of possible determinants of scalar matrices in $\GL_n(\F)$.

    \begin{proposition}\label{prop:PSLlower}
      Assume that $n \ge 4$ and $|\F| \geq 3$, 
      and let $\gp$ satisfy $\SL_n(\F) \trianglelefteq \gp \le \GamL_n(\F)$. Assume also that the set $\{\det(g)^\psi \F^{\times n} \mid  g \psi \in H$ \text{with} $g \in \GL_n(\F), \psi \in \Aut(\F)\}$ is a proper subset of $\F^\ast/\F^{\times n}$. 
    Then $\RC (\gp) \geq 2n-2.$
\end{proposition}

\begin{proof}
By assumption, there exists an $\alpha \in \F^\ast$ such that $\alpha \neq \det(gz)^\psi$ for all $g\psi \in H$ and $z \in Z$.
  Define $\X, \Y \in \Omega^{2n-2}$ as follows:
\begin{align*}
\X &:= \big(\langle e_2 \rangle, \ldots, \langle e_{n} \rangle, \langle e_1+e_2 \rangle, \ldots ,\langle e_1+e_{n} \rangle \big); \text{ and}\\
\Y &:= \big(\langle e_2 \rangle, \ldots, \langle e_{n} \rangle, \langle \alpha e_1+e_2 \rangle, \ldots , \langle \alpha e_1+e_{n} \rangle \big).
\end{align*}
We show first that $\Y \not\in \X^\gp$, so
suppose for a contradiction that there exists $g\psi \in \gp$,
with $g \in \GL_n(\F)$ and $\psi \in \Aut(\F)$, such that $\X^{g\psi} = \Y$. As $g\psi$ fixes $\langle e_2 \rangle$ and $\langle e_3 \rangle$, and maps $\langle e_1+e_2 \rangle $ and $\langle e_1+e_3 \rangle $ to $\langle \alpha e_1+e_2 \rangle$ and $\langle \alpha e_1+e_3 \rangle$, respectively, we deduce that $e_1^{g\psi} \in \langle e_1, e_2 \rangle \cap \langle e_1, e_3 \rangle = \langle e_1 \rangle$. Therefore $\langle e_i \rangle^{g\psi}=\langle e_i \rangle$ for each $i \in \{1, \ldots, n\}$, and so $g$ is  diagonal.
Let $\mu:=\alpha^{\psi^{-1}}$. As $\langle e_1+e_{i} \rangle ^{g\psi}
= \langle \alpha e_1 + e_{i} \rangle$ for each $i  \in \{2, \ldots,
n\}$, we deduce that $g =\diag (\mu, 1, \ldots, 1) z$
  for some $z \in Z$. Hence $(\det (g z^{-1}))^\psi = \mu^\psi  = \alpha$, a contradiction.
Hence
$\Y \not\in \X^\gp$.

Now, for each $i  \in \{2, \ldots, n\}$, let $h_{i} := \diag(\alpha, 1,
\ldots, 1, \alpha^{-1}, 1, \ldots, 1)$, where the $\alpha^{-1}$
appears in entry  $i$.
  First, for  $j \in \{1, \ldots, n-1\}$, let $k:=j+1$ so that $\x_{j} = \y_{j} = \langle e_{k} \rangle$.   It is easy to verify that $h_{k} + (1-\alpha)E_{k1}$ has determinant $1$ and 
maps $(\X \setminus \x_{j}) $ to $(\Y \setminus \y_{j})$. Finally, for $j \in \{n,\ldots, 2n-2\}$, let $k:=j+2-n$, so that $x_{j} = \langle e_1 + e_{k} \rangle$ and $\y_{j} = \langle \alpha e_1 + e_{k} \rangle$. Then $h_{k}$ has determinant $1$ and
maps $(\X \setminus \x_{j}) $ to $(\Y \setminus \y_{j})$. Therefore, $\X \s{2n-3}{\gp} \Y$, and so $\RC (\gp) \geq 2n-2$.
\end{proof}

\begin{proof}[Proof of Theorem~\ref{thm:A}]
When $|\F|= 2$, this result is clear from
Theorem~\ref{thm:CherlinNonzero}. For the remaining fields $\F$, the fact that Part~\ref{A:PGL} gives an upper bound on $\RC(\PGL_n(\F))$ is proved in Theorem~\ref{thm:relupper}, whilst we prove that it gives a lower bound in Theorem~\ref{thm:lower_bound_n} for $|\F| = 3$ and
Proposition~\ref{prop:PGLlow} for $|\F| \geq 4$. That Part~\ref{A:notPGL} gives upper bounds on $\RC(\ogp)$ is immediate from Theorem~\ref{prop:lodabounds}\ref{proppart:bianca} for $n = 3$, and from Theorem~\ref{thm:PSLupper} for $n \ge 4$. Lemma~\ref{lem:PSL-n=3} and Proposition~\ref{prop:PSLlower} show that these are also lower bounds.
\end{proof}

Recall that $\omega(k)$ denotes the number of distinct prime divisors of the positive integer $k$.

\begin{lemma}[{\cite[Lemma 3.1]{Harper}}] \label{lem:Harper} Let $K \leq \Sym{\Gamma}$ be a finite group with normal subgroup $N$ such that $K/N$ is cyclic. Then $\H(K,\Gamma)\leq \H(N,\Gamma)+\omega(|K/N|)$.
\end{lemma}

\begin{proof}[Proof of Theorem~\ref{thm:B}]
For the upper bound in Part~\ref{B:n=2}, we combine Proposition~\ref{prop:lodabounds}\ref{proppart:scott} with Lemma~\ref{lem:Harper}
to deduce that $\H(\ogp, \Omega_1) = 3 + \omega(e)$, so $\RC(\ogp,
\Omega_1) \le 4 + \omega(e)$.  The lower bound (and the case $\ogp = \PSigL_2(9)$) is Lemma~\ref{lem:PGamL2(q)}.

For the upper bound in Part~\ref{B:ngeq3}, we similarly combine Proposition~\ref{prop:lodabounds}\ref{proppart:bianca} with Lemma~\ref{lem:Harper}. As for the lower bound, first let $n = 3$, and notice that in this case $2n-2 = 4 < n+2 = 5$. If $\ogp$ properly contains $\PGL_3(q)$, then the lower bound of $6$ is proved in Lemma~\ref{lem:PGamLn(q)}. Otherwise, $\PSL_3(q) \ne \PGL_3(q)$, and so the lower bound of $5$ follows from Lemma~\ref{lem:PSL-n=3}.
Now assume that $n \ge 4$. The general lower bound is Lemma~\ref{lem:general_n+2}, the bound of $n+3$ for groups properly containing $\PGL_n(q)$ is Lemma~\ref{lem:PGamLn(q)}, and the bound of $2n-2$ is Proposition~\ref{prop:PSLlower}.
\end{proof}

\section{Action on $m$-spaces for $m \ge 2$}\label{sec:mspace}

In this section, we
consider the
action of $\gp$ on $\Omega_m=\PG_m(V)$, where $\SL_n(\F)
\trianglelefteq \gp \le \GamL_{n}(\F)$, as before, but now $2
\leq m  \le \frac{n}{2}$. The main work is to prove a lower
  bound on $\RC(H, \Omega_m)$, as the upper bound follows from
  existing literature.

\begin{proposition}
\label{prop:lowerm}
  Let $\F$ be arbitrary, let $n \ge 2m \ge 4$, and let $H$
  satisfy $\SL_n(\F) \trianglelefteq H \le \GamL_n(\F)$. Then $\RC(H, \Omega_m) \ge mn-m^2+1$.
\end{proposition}

\begin{proof}
  For each $i \in \{1, \ldots,m\}$ and $j \in
\{m+1, \ldots, n-1\}$, let $\BB_i := \{e_1, e_2, \ldots, e_{m}\}
\setminus \{e_i\}$, $U_{ij} :=
\langle B_i, e_j \rangle = \langle e_1, \ldots, e_{i-1}, e_{i+1},
\ldots, e_m, e_j\rangle$,  $V_i := \langle B_i, e_i + e_n
\rangle$, and $W_i := \langle B_i, e_n \rangle$, so that $U_{ij}, V_i, W_i \in \Omega_m$. Define $\X, \Y \in \Omega_m^{mn-m^2+1}$ as follows:
\begin{equation*}
  \begin{aligned}
    & x_{mn - m^2 + 1}  := \langle e_1+e_2, \ldots, e_1+e_m,
    \sum_{i=1}^{n} e_i \rangle;\\
    & y_{mn - m^2 + 1} := \langle
  e_1+e_2, \ldots, e_1+e_m, -e_1+\sum_{i=m+1}^n e_i \rangle;\\
  & \X := \Big( U_{1 (m+1)}, U_{1 (m+2)}, \ldots, U_{m (n-1)}, 
  V_1, V_2, \ldots, V_m, x_{mn-m^2+1}
  \Big); \text{ and}   \\
  &\Y := \Big(
U_{1 (m+1)}, U_{1 (m+2)}, \ldots, U_{m (n-1)},  W_1, W_2, \ldots, W_m, 
y_{mn-m^2+1}\Big).
\end{aligned}
\end{equation*}
We shall first show that $Y \not\in \X^{\GamL_n(\F)}$, so in particular $Y \not\in \X^{\gp}$, and then that $\X \s{mn-m^2}{H} \Y$.
    
    Assume for a contradiction that $\Y \in
  \X^{\GamL_n(\F)}$. Since each subspace in $\Y$ is spanned by 
  vectors of the form $\sum_{i=1}^n \lambda_i e_i$ with $\lambda_i \in
  \{-1,0,1\}$, it follows that there exists $g \in \GL_n(\F)$ with $\X^g =
  \Y$. For each $i \in \{1, \ldots,m\}$, choose $k \in \{1, \ldots, m\} \setminus \{i\}$. Then  
  \[\langle e_i \rangle = \bigcap\limits_{\ell \in \{1, \ldots, m\} \setminus \{i\}}
  U_{\ell(m+1)} \cap 
  V_k   = \bigcap\limits_{\ell \in \{1, \ldots, m\} \setminus \{i\}}
  U_{\ell(m+1)} \cap W_k,\] so  $g$ fixes $\langle e_i \rangle$. 
  Similarly, $g$ fixes $\langle e_j \rangle = \bigcap\limits_{i=1}^{m} U_{ij}$
  for each $j \in \{m+1,\ldots,n-1\}$.

Therefore, there exist $\lambda_1,\ldots,\lambda_{n} \in \F^*$ and
$\mu_1, \ldots, \mu_{n-1} \in \F$ such that  $g$ maps $e_i$ to
$\lambda_i e_i$ for all %
$i \in \{1, \ldots, n-1\}$, and maps $e_n$ to $\lambda_n e_n + \sum_{i =
  1}^{n-1} \mu_i e_i$. 
It now follows that for each $i \in \{2, \ldots, m\}$, the element $g$
maps $e_1+e_i \in \x_{mn-m^2+1}$ to $\lambda_1 e_1 + \lambda_i
e_i$, which must lie in $\y_{mn - m^2+1}$, and hence %
$\lambda_i = \lambda_1$. Similarly,
$V_i^g = W_i$ for each $i \in \{1, \ldots, m\}$, and so $W_i = \langle \BB_i, e_{n} \rangle$ contains $$(e_i+e_{n})^g = \lambda_1 e_i +\lambda_{n}e_{n}+ \sum_{k=1}^{n-1} \mu_k e_k.$$
Hence $\mu_i = - \lambda_1$, and $\mu_j =0$ for all $j \in \{m+1, \ldots, n-1\}$. It now follows that $g$ maps $\sum_{i=1}^{n} e_i \in \x_{mn-m^2+1}$ to $\sum_{i=m+1}^{n} \lambda_i e_i$, which is clearly not in $\y_{mn-m^2+1}$, a contradiction. Thus $\Y \not\in \X^{\gp}$.

\medskip

We now show that $\X \s{mn-m^2}{H} \Y$, by identifying an element $g_\ell \in \SL_n(\F) \leq H$ that
maps $(\X \setminus \x_{\ell})$ to $(\Y \setminus \y_{\ell})$, for each $\ell \in \{1, \ldots, mn-m^2+1\}$. We divide the proof into
three cases, which together account for all values of $\ell$. To simplify our expressions, let $z := e_1 + e_2 + \ldots + e_m$, $\alpha_1:=-1$, and $\alpha_r:=1$ for all $r \in \{2,\ldots,m\}$. In each case the element $g_{\ell}$ will be lower unitriangular, and so will have determinant $1$.

\medskip

\noindent \textbf{Case (a)}: $\ell \in \{1, \ldots, m(n-m-1)\}$. Let $r\in \{1, \ldots, m\}$ and $s\in \{m+1, \ldots, n-1\}$ be such that $\ell =  (n-m-1)(r-1)+(s-m)$, so
that
$\x_\ell = \y_\ell = U_{rs}$. %
Additionally, let $g_\ell$ fix  $e_i$ for all $i \notin \{s,n\}$, map $e_s$ to
$e_s+\alpha_r e_r$, and map $e_n$ to $e_n -  z$. %
Then $g_{\ell}$
fixes $U_{ij}$ provided $(i, j) \neq (r, s)$, and maps $e_i + e_n \in V_i$
to $e_i +e_n - z \in W_i$, and hence $V_i$ to $W_i$, for all $i \in \{1, \ldots m\}$. %
Finally, 
\[\Big(\sum_{i=1}^{n} e_i \Big)^{g_\ell} =  \alpha_r e_r +\sum_{i=m+1}^{n} e_i
  \in  \y_{mn-m^2+1}, \]
where we have used the fact that $e_r +\sum_{i=m+1}^{n} e_i = (e_1+e_r)+(-e_1+ \sum_{i=m+1}^{n} e_i)$ when $r > 1$. Hence $g_{\ell}$ maps $x_{mn-m^2 + 1}$ to $y_{mn-m^2 + 1}$, as required. 

\medskip

\noindent \textbf{Case (b)}: $\ell = {m(n-m-1)+r}$, where $r \in \{1, \ldots, m\}$. Here, $\x_\ell =
V_r$ and $\y_\ell = W_r$. Let $g_\ell$ fix $e_i$ for each $i \in \{1,
\ldots, n-1\}$ and map $e_n$ to $\alpha_r e_r + e_n - z$. 
Then $g_\ell$ fixes $U_{ij}$ for all $i$ and $j$, %
and maps $e_i + e_n \in V_i$ to $e_i + \alpha_r e_r + e_n - z \in W_i$,
and hence $V_i$ to $W_i$, for all $i \in \{1, \ldots, m\} \setminus \{r\}$. Finally,
$$\Big( \sum_{i=1}^{n} e_i \Big)^{g_\ell} = \alpha_r e_r+\sum_{i=m+1}^n e_i  \in \y_{mn-m^2+1},$$
as in Case (a).

\medskip 

\noindent \textbf{Case (c)}: $\ell={mn-m^2+1}$. Let %
$g_\ell$ fix $e_i$ for each $i \in \{1, \ldots, n-1\}$, and map $e_n$ to $e_n - z$. Then $g$ fixes $U_{ij}$ for all $i$ and $j$, and maps $e_i + e_n \in V_i$ to $e_i + e_n - z \in W_i$
for all $i$, as required. %
\end{proof}

The \emph{irredundant base size} $\I(K, \Gamma)$ of a group $K$ acting faithfully on a set $\Gamma$ is
  the largest size of a tuple $(\alpha_1,\ldots,\alpha_k)$ 
  of elements of $\Gamma$ such that $K > K_{\alpha_1} > K_{(\alpha_1,
      \alpha_2)} > \cdots > K_{(\alpha_1, \ldots, \alpha_k)} = 1$,
  with all inclusions strict. It is clear that $\I(K, \Gamma)$
  is bounded below by the height $\H(K,\Gamma)$, which we recall (from
  \S\ref{sec:intro}) is bounded below by $\RC(K,\Gamma)-1$.
  
\begin{proof}[Proof of Theorem~\ref{thm:Bm}] In \cite[Thm 3.1]{KRD}, it is proved that $\I(\PGL_{n}(\F), \Omega_m) \le (m+1)n - 2m +1$. Since the irredundant base size of a subgroup is at most the irredundant base size of an overgroup, and the height is at most the irredundant base size, we deduce that $\H(\ogp, \Omega_m) \le (m+1)n-2m+1$ for all $\ogp \le \PGL_n(\F)$. From Lemma~\ref{lem:Harper}, we then see that for all $\ogp$ as in the statement, $\H(\ogp, \Omega_m) \le (m+1)n - 2m + 1 + \omega(e)$, and hence the upper bound follows. The lower bound is immediate from Proposition~\ref{prop:lowerm}, so the proof is complete.
\end{proof}

{\small \noindent \textbf{Acknowledgments}
We thank the anonymous referee for their careful reading and helpful
comments.
The authors would like to thank the Isaac Newton Institute for
Mathematical Sciences for support and hospitality during the programme
``Groups, representations and applications: new perspectives'', when work
on this paper was undertaken.  This work was supported by EPSRC grant
no. EP/R014604/1, and also partially supported by a grant from the
Simons Foundation. The first author was supported by the University of
St Andrews (St Leonard's International Doctoral Fees Scholarship \&
School of Mathematics and Statistics PhD Funding Scholarship), and by
EPSRC grant no. EP/W522422/1. The second author is funded by the
Heilbronn Institute. \newline In order to meet institutional and
research funder open access requirements, any accepted manuscript
arising shall be open access under a Creative Commons Attribution (CC
BY) reuse licence with zero embargo.}

\smallskip

{\small \noindent \textbf{Competing interests}
The authors declare none.}
  
\bibliographystyle{plain}
\bibliography{RCrefs}

\end{document}